\documentclass[12pt,twoside]{amsart}
\usepackage{graphicx,color}

\theoremstyle{plain}
\newtheorem{The}{Theorem}
\newtheorem*{The*}{Theorem}
\newtheorem{Pro}{Proposition}

\newtheorem*{Cor*}{Corollary}
\theoremstyle{definition}

\newtheorem{Rem}{Remark}

\newtheorem*{Rem*}{Remark}
\numberwithin{equation}{section}

\newcommand{\R}{\mathbb{R}}
\newcommand{\C}{\mathbb{C}}
\newcommand{\N}{\mathbb{N}}
\newcommand{\Z}{\mathbb{Z}}

\newcommand{\CP}{\mathbb{CP}}

\DeclareMathOperator{\End}{End}

\DeclareMathOperator{\SL}{SL}

\DeclareMathOperator{\SU}{SU}
\DeclareMathOperator{\Tr}{tr}

\DeclareMathOperator{\pdeg}{pdeg}

\renewcommand{\Im}{\operatorname{Im}}
\renewcommand{\Re}{\operatorname{Re}}
\newcommand{\dvector}[1]{{\left(\begin{matrix}#1\end{matrix}\right)}}
\newcommand{\bbR}{\mathbb{R}}
\newcommand{\bbS}{\mathbb{S}}
\newcommand{\bbC}{\mathbb{C}}
\newcommand{\Cstar}{{\bbC^\times}}

\begin{document}

\title{Deformations of symmetric CMC surfaces in the 3-sphere}

\author{Sebastian Heller}

\author{Nicholas Schmitt}

\address{Sebastian Heller \\
Institut f\"ur Mathematik\\
Universit{\"a}t T\"ubingen\\ Auf der Morgenstelle
10\\ 72076 T¬ubingen\\ Germany
}

\address{Nicholas Schmitt \\
Institut f\"ur Mathematik\\
Universit{\"a}t T\"ubingen\\ Auf der Morgenstelle
10\\ 72076 T¬ubingen\\ Germany
}

\email{heller@mathematik.uni-tuebingen.de}
\email{nschmitt@mathematik.uni-tuebingen.de}

\date{\today}

\begin{abstract} 
We numerically construct CMC deformations of the Lawson minimal surfaces $\xi_{g,1}$
using a DPW approach to CMC surfaces in spaceforms.
\end{abstract}

\maketitle

\section*{Introduction}
\label{intro}

The study of minimal surfaces in 3-dimensional space forms is among
the oldest subjects in differential geometry.  While minimal surfaces
in euclidean 3-space are never compact, there exist compact minimal
surfaces in $\bbS^3$. In fact, Lawson~\cite{L} has shown that there exist
embedded closed minimal surfaces in the 3-sphere, known as
the Lawson surfaces $\xi_{p,q}$ of genus $g=pq$ for $p,q\in\N$. A
slightly more general surface class is that of constant mean
curvature (CMC) surfaces.  Due to the Lawson correspondence, the
partial differential equations describing minimal and CMC surfaces in
$\bbS^3$ and $\bbR^3$ can be treated in a uniform way.  Compact minimal and
CMC surfaces in $\bbS^3$ of genus $0$ and $1$ are well understood:
because CMC spheres have vanishing Hopf differential they must
be totally umbilic; similarly, CMC tori have constant Hopf differential and thus cannot have umbilic points. 
 Brendle~\cite{B} recently proved the Lawson
conjecture, that the only embedded minimal torus in the 3-sphere is the
Clifford torus. Using Brendle's ideas, Andrews and Li~\cite{AL} have
classified all embedded CMC tori in $\bbS^3$. Additionally, all CMC
immersions from a torus into a $3$-dimensional spaceform are given
rather explicitly in terms of algebro-geometric data of their
associated spectral curves~\cite{PS, Hi2, Bo}.  These integrable
system methods also apply to the study of the moduli space of all CMC
tori, see for example~\cite{KS,KSS}. 

Besides the Lawson surfaces, there also exist other embedded (minimal)
examples in $\bbS^3$~\cite{KPS} as well as immersed CMC examples in
$\R^3$~\cite{K1,K2} for all genera, constructed by
implicit methods from geometric analysis.  Methods from geometric
analysis have also been applied to the study of the moduli space of CMC
surfaces of higher genus. For example, it follows from the work of Kusner,
Mazzeo and Pollack~\cite{KMP} that the moduli space of embedded CMC
surfaces in $\bbS^3$
(without fixing the value of the mean curvature)
is a real analytic variety of formal dimension $1$.
Nevertheless, the global structure of the space of embedded CMC
surfaces in $\bbS^3$ is not well understood. It is also not clear whether
there exist higher-dimensional families of compact CMC surfaces of
genus $g \geq 2$ (with fixed mean curvature) analogous to the
isospectral families of (immersed but not embedded)
CMC tori parametrized by the real part of the
Jacobian of their spectral curve.

In this paper, we report on the experimental construction of families
of compact embedded higher genus CMC surfaces in $\bbS^3$ based on
recently developed integrable system methods~\cite{He1,He2,He3}.  This
approach is based on the associated family of flat
$\text{SL}(2,\mathbb{C})$-connections
$\nabla^\lambda$ \eqref{associated_family} on a fixed hermitian rank
$2$ bundle~\cite{Hi2}.  The connections $\nabla^\lambda$ are unitary
for $\lambda\in \bbS^1\subset\Cstar$ and trivial at two \emph{Sym
points} $\lambda_1\neq\lambda_2\in \bbS^1$. The immersion is recovered
from $\nabla^\lambda$ as the gauge between $\nabla^{\lambda_1}$ and
$\nabla^{\lambda_2}$, and its mean curvature is
$H=i\frac{\lambda_1+\lambda_2}{\lambda_1-\lambda_2}$.

The direct construction of the associated family of flat connections
of a CMC surface is difficult. Easier is the construction of a related
family of flat connections $\tilde\nabla^\lambda$ which determines
$\nabla^\lambda$ by loop group factorization methods.  The family
$\tilde\nabla^\lambda$ must satisfy weaker conditions only, for
example the connections $\tilde\nabla^\lambda$ are only unitarizable,
that is, unitary with respect to a $\lambda$-dependent metric along the
unit circle (see theorem~\ref{general_reconstruction}).  The DPW
approach to (simply-connected) CMC surfaces~\cite{DPW} employes this
insight to construct CMC surfaces in terms of a so-called DPW
potential, a loop of meromorphic $\mathfrak{sl}(2,\C)$-valued
connection $1$-forms.  The DPW construction was adapted to compact CMC
surfaces of genus $g\geq2$ in~\cite{He1,He2}.

For compact CMC surfaces of genus $g\geq2$, the connections
$\tilde\nabla^\lambda$ must be irreducible for generic
$\lambda\in\Cstar$ (see~\cite{He1}). Hence it not known
how to compute the corresponding family of monodromy representations.
Note that a flat connection is unitarizable if and only if its
monodromy representation is unitary modulo conjugation.  Thus it
is not known how to construct appropriate families of flat connections
$\tilde\nabla^\lambda$ which are unitarizable along the unit circle.
But this intrinsic closing condition can be tackled numerically: using
numerical ODE solvers one can compute the monodromy representation to
determine whether a connection is unitarizable.  In order to construct
compact CMC surfaces numerically we have implemented a minimizer which
approximates families of unitarizable DPW potentials.  We have carried
out the DPW experiments for Lawson symmetric surfaces.  These are
equipped with a large group of ambient
orientation preserving
symmetries.  One feature of these embeddings is that their
curvature lines are closed (figure~\ref{fig:lawson}).

\begin{figure}
\centering
\includegraphics[width=0.495\textwidth]{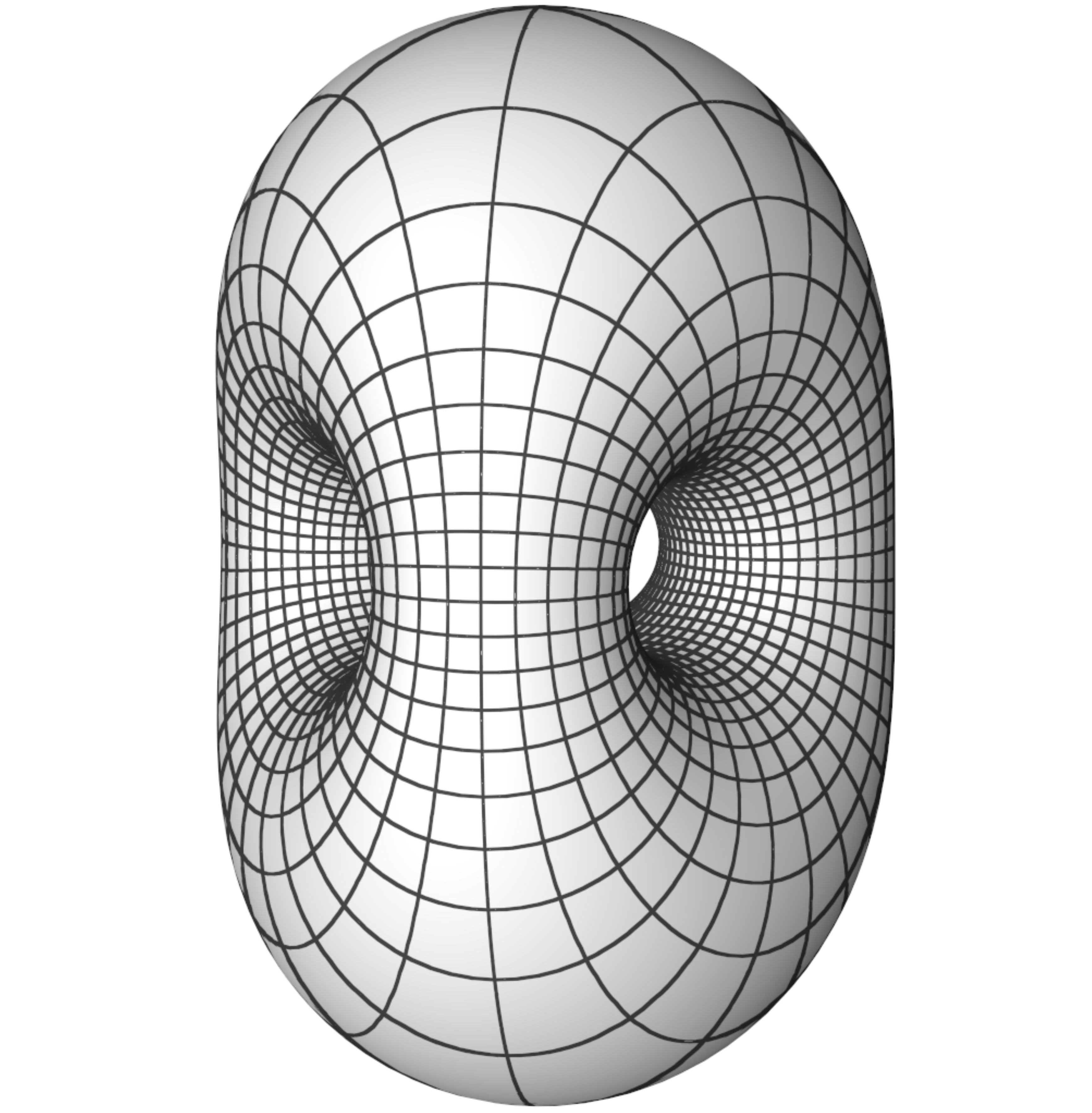}
\includegraphics[width=0.495\textwidth]{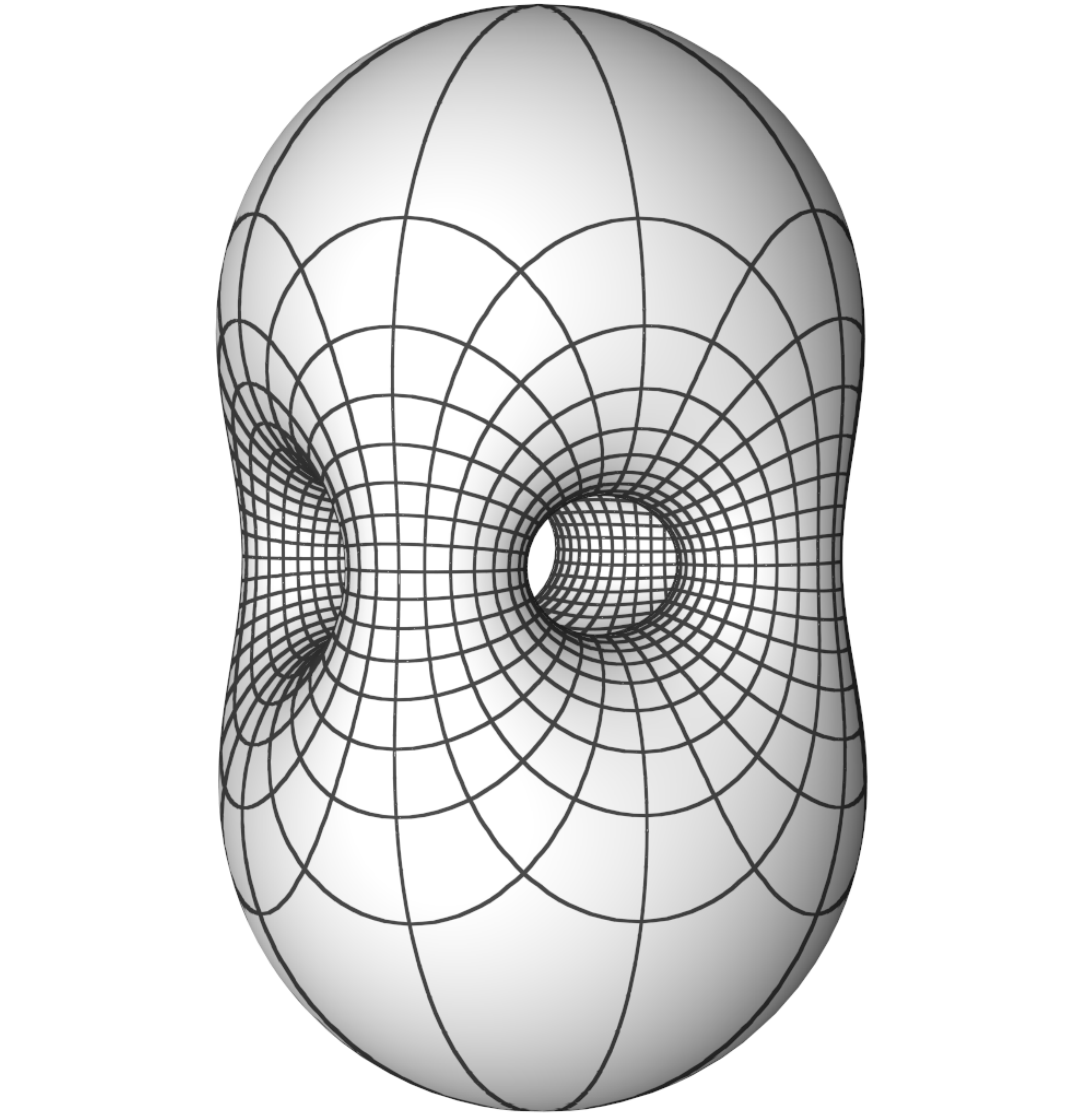}
\vspace{-0.2cm}
\includegraphics[width=0.495\textwidth]{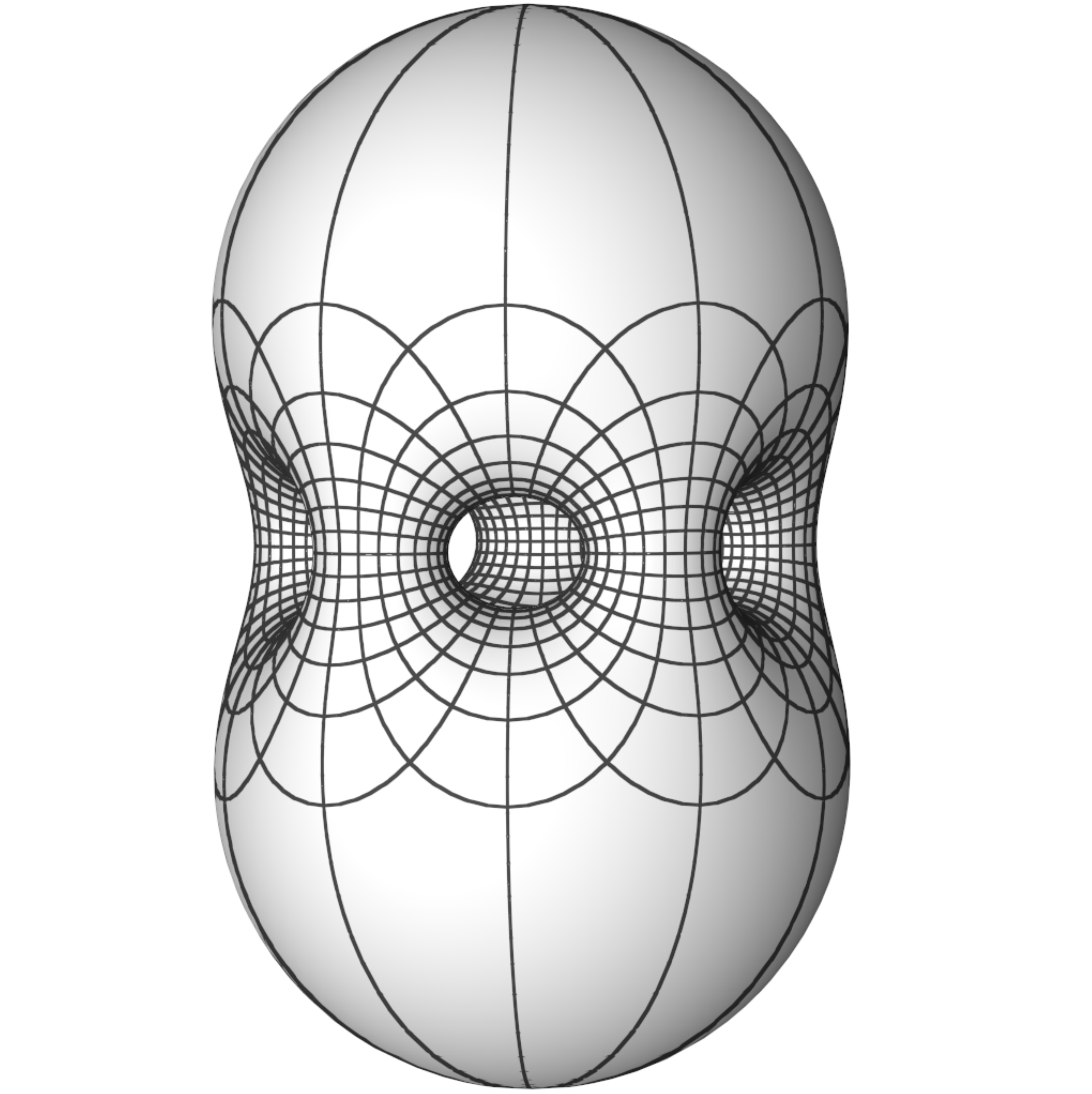}
\includegraphics[width=0.495\textwidth]{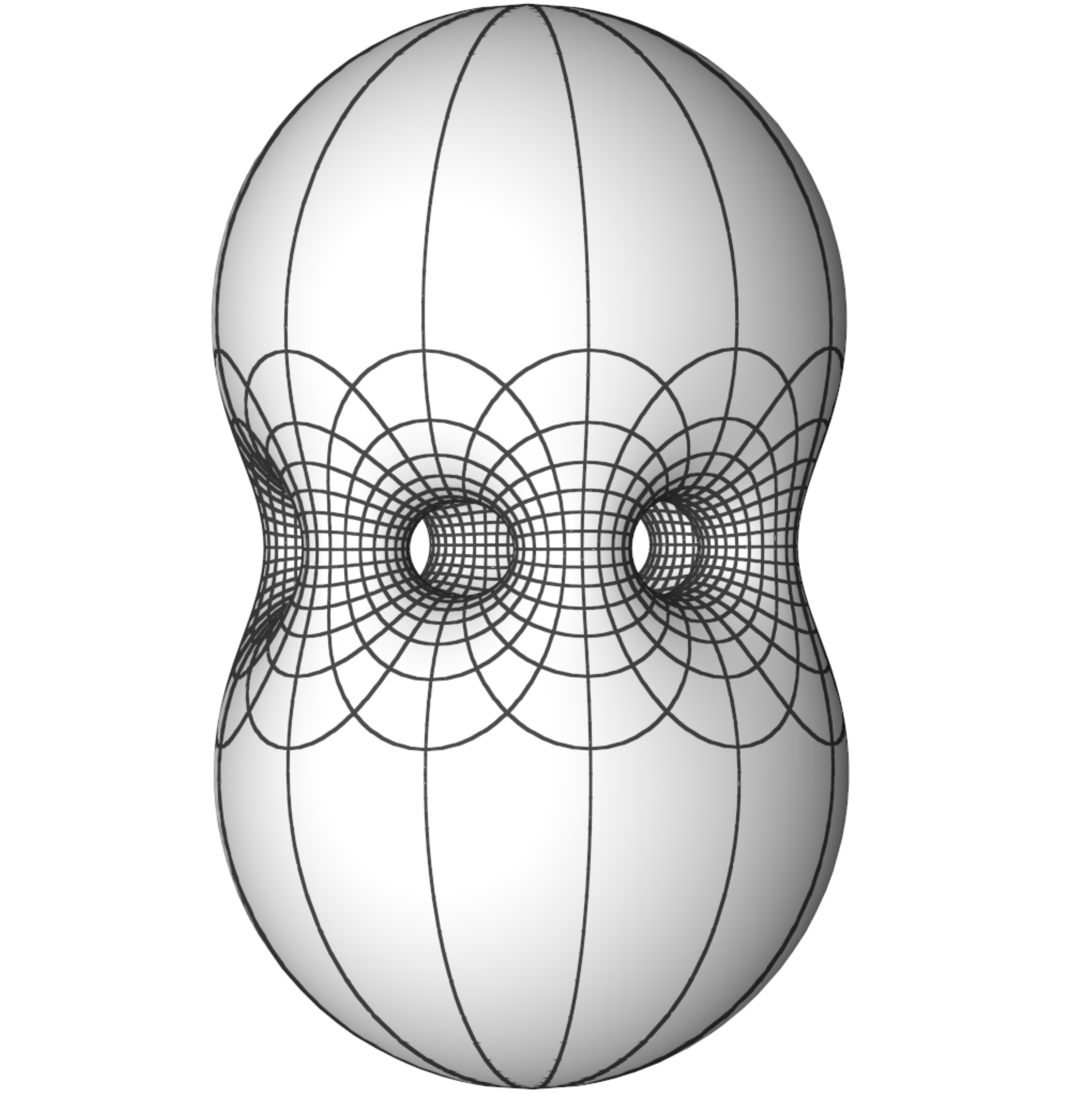}
\caption{
\footnotesize
The Lawson minimal surfaces $\xi_{g,1}$ in $\bbS^3$ of genus $g=2,\,3,\,4,\,5$.
All surfaces are shown are steregraphically projected from $\bbS^3$ to $\bbR^3$
and are depicted with curvature lines.
}
\label{fig:lawson}
\end{figure}

In our experiments, we realized the expected $1$-dimensional family of
CMC deformations of the Lawson surfaces $\xi_{g,1}$ as numerically
computed conformal embeddings (figure~\ref{fig:lawsoncmc1}).  In
the case of $g=1$ this family consists of the well-known homogeneous
CMC tori.  The numerical computation of this $1$-parameter family of
CMC tori via the (non-abelian) DPW method serves as a test for our
method.  We also observed the bifurcation of the homogeneous CMC tori
into the family of 2-lobed Delaunay tori of spectral genus 1 and the
continued family of homogeneous CMC tori~\cite{KSS}.  For higher genus
Lawson symmetric CMC surfaces such bifurcations did not appear; these
families continue until they collapse into double coverings of minimal
spheres (as do the Delaunay tori)
while their conformal types degenerate. In genus $2$ we also found a second
family of Lawson symmetric CMC surfaces which converges to a 
necklace of three CMC spheres (figure~\ref{fig:lawsoncmc2a}).  This
second family serves as the analog to the Delaunay family in the case
of genus $2$ surfaces but it is disjoint from the first family passing
through $\xi_{2,1}$. Altogether, our experiments begin to map out the
moduli space of Lawson symmetric CMC surface of genus $2$
(figure~\ref{fig:graph}).
It is worth noting that none of
our numerically constructed surfaces, if they indeed exist, would
allow for isospectral deformations --- deformations of CMC surfaces
such that the corresponding flat connections of their associated
families are gauge equivalent for generic $\lambda\in\Cstar$
(section~\ref{outlook}).  We expect that
higher-dimensional families of compact CMC surfaces of higher genus
are isospectral deformations, as is the case for tori.

As was pointed out to us by the anonymous reviewer,
existence of further Lawson symmetric CMC surfaces of genus $2$ in
$\bbS^3$ can be deduced
from gluing methods by Kapouleas and related analytic methods.
We expect such surfaces to exist for more than one unstable or strictly semi-stable parabolic structure inside the unit
disc (see section \ref{gen_spec_gen}).
We plan to construct such surfaces by investigating the relationship between
the spectral curve theory for tori~\cite{Hi2}
and for Lawson symmetric CMC surfaces of genus $2$~\cite{He3}
by implementing a flow between their DPW potentials.

The paper is organized as follows. In section~\ref{theory} we recall
the necessary theory for our experiments.  In section~\ref{sec:experiment} we
discuss the computational aspects of our studies and report on the
experiments concerning Lawson symmetric CMC surfaces. In the last
section~\ref{outlook} we give a short theoretical outlook.

The authors would like to thank the anonymous referee for helpful comments. 
Both authors were supported by the DFG project grants HE 6829/1-1
and PE 1535/6-1.

\section{Theoretical background}
\label{theory}

 \subsection{The associated family of flat connections}
We briefly recall the gauge-theoretic description of conformal CMC
immersions $f\colon M\to \bbS^3$ from a Riemann surface into the round
3-sphere~\cite{Hi2,Bo,He1}. Due to the Lawson correspondence, there is
a unified treatment for all mean curvatures $H\in\bbR$:

\begin{The}
\label{The1}
Let $f\colon M\to \bbS^3$ be a conformal CMC immersion. Then there exists
an associated family of flat $\SL(2,\C)$ connections
\begin{equation}\label{associated_family}
\lambda\in\Cstar\mapsto \nabla^\lambda=\nabla+\lambda^{-1}\Phi-\lambda\Phi^*
\end{equation}
on a hermitian rank $2$ bundle $V\to M$. The family $\nabla^\lambda$
is unitary along $\bbS^1\subset\Cstar$ and trivial at
$\lambda_1\neq\lambda_2\in \bbS^1$. Here, $\Phi$ is a nowhere vanishing
complex linear $1$-form which is nilpotent and $\Phi^*$ is its
adjoint.  Conversely, the immersion $f$ is given as the gauge between
$\nabla^{\lambda_1}$ and $\nabla^{\lambda_2}$ where we identify
$\SU(2)=\bbS^3$, and its mean curvature is
$H=i\frac{\lambda_1+\lambda_2}{\lambda_1-\lambda_2}$. Therefore, every
family of flat $\SL(2,\C)$ connections satisfying the conditions above
determines a conformal CMC immersion.
\end{The}

It is well known~\cite{Hi2} that for compact CMC surfaces which are
not totally umbilic the generic connection $\nabla^\lambda$ of the
associated family is not trivial. Moreover, for CMC immersions from a
compact Riemann surface of genus $g\geq2$, the generic connection
$\nabla^\lambda$ of the associated family is irreducible~\cite{He1}.

Fundamental for our approach to CMC surfaces is the following
generalization of theorem~\ref{The1}:
\begin{The}\label{general_reconstruction}
Let $U\subset \C$ be an open set containing the disc of
radius $1+\epsilon$ centered at $0$. Let $\lambda\in
U\setminus\{0\}\mapsto\tilde\nabla^\lambda$ be a holomorphic family of
flat $\SL(2,\C)$ connections on a rank $2$ bundle $V\to M$ over a
compact Riemann surface $M$ of genus $g\geq2$ such that
\begin{itemize}
\item the asymptotic at $\lambda=0$ is given by
\[\tilde\nabla^\lambda\sim \lambda^{-1}\Psi+\tilde\nabla+...\]
where $\Psi\in\Gamma(M,K\End_0(V))$ is nowhere vanishing and
nilpotent;
\item the intrinsic closing condition holds: for all $\lambda\in
\bbS^1\subset U\subset \C$ there is a hermitian metric on $V$ such that
$\tilde\nabla^\lambda$ is unitary with respect to this metric;
\item the extrinsic closing condition holds: $\tilde\nabla^\lambda$ is
trivial for $\lambda_1\neq\lambda_2\in \bbS^1$.
\end{itemize}
Then there exists a unique CMC surface $f\colon M\to \bbS^3$ of mean
curvature $H=i\frac{\lambda_1+\lambda_2}{\lambda_1-\lambda_2}$ such
that its associated family of flat connections $\nabla^\lambda$ and
the family $\tilde\nabla^\lambda$ are gauge equivalent, i.e., there
exists a $\lambda$-dependent holomorphic family of gauge
transformations $g$ which extends through $\lambda=0$ such that
$\nabla^\lambda\cdot g(\lambda)=\tilde\nabla^\lambda$ for all
$\lambda\in U\setminus\{0\}$.
\end{The}

In the above form, this theorem was proven in~\cite{He3}, but there
are earlier variants which are more adapted to the DPW approach to
$k$-noids~\cite{SKKR,DW}.
\begin{Rem}\label{generalized_reconstruction}
Theorem~\ref{general_reconstruction} remains true if there exist
$\lambda$-independent apparent singularities of the connections
$\tilde\nabla^\lambda$. It also remains true if there are finitely
many points on the unit circle at which the monodromy fails to be
unitarizable (this situation can occur when the gauge transformation
between $\nabla^\lambda$ and $\tilde\nabla^\lambda$ has apparent
singularities in $\lambda$).  In both cases, the corresponding
singularities (on the Riemann surface in the first case, and on the
spectral plane in the second) are captured in the positive part of the
Iwasawa decomposition (see the proof of this theorem
in~\cite{He3}). Therefore, the actual associated family of flat
connections $\nabla^\lambda$ has no singularities anymore, and the CMC
immersion is well-defined.
\end{Rem}

\subsection{Lawson symmetric CMC surfaces}\label{LSS}
From now on we focus on CMC immersions from compact Riemann surfaces
of genus $2$ which have the following extrinsic symmetries:
\begin{itemize}
\item an involution $\varphi_2$ with $6$ fixed points which commutes with the 
other symmetries;
\item a $\Z_3$ symmetry generated by $\varphi_3$ with $4$ fixed points;
\item an involution $\tau$ with $2$ fixed points;
\item
optionally, a reflection across a great 2-sphere perpendicular to the
surface.
\end{itemize}
Surfaces with the first three symmetries,
which preserve the orientation of $\bbS^3$ and the surface,
are called \emph{Lawson symmetric CMC surfaces} (of genus $2$),
and those with all four symmetries are called \emph{rectangular Lawson symmetric surfaces.}

Due to the symmetry breaking of a stereographic projection it is not
possible to visualize all symmetries simultaneously in euclidean 3-space.
One can see the $\Z_3$ symmetry of the surfaces in
figure~\ref{fig:lawsoncmc1} and figure~\ref{fig:lawsoncmc2a} and the
symmetries $\varphi_2$ and $\tau$ in
figure~\ref{fig:lawsoncmc2b}.  The Lawson symmetries already
fix the Riemann surface structure up to one complex parameter.
 All
Lawson symmetric Riemann surfaces are given by an algebraic equation
of the form
\begin{equation}
\label{rss}
y^3=\frac{z^2-z_0^2}{z^2-z_1^2};
\end{equation}
see \cite{He2} for more details.
Lawson symmetric Riemann surfaces corresponding to tuples
$(z_0,z_1,-z_0,-z_1)$ with the same cross-ratio are isomorphic.  The
Riemann surface structure of the Lawson surface $\xi_{2,1}$ is given
by $z_0=1$, $z_1=i$. In terms of the coordinates $(y,z)$ the
symmetries of a Lawson symmetric Riemann surface are given by
$\varphi_2(y,z)=(y,-z)$, $\varphi_3(y,z)=(e^{\frac{2}{3}\pi i}y,z)$ and
$\tau(y,z)=((\frac{z_0}{z_1})^{\frac{2}{3}}\frac{1}{y},\frac{z_0 z_1}{z})$. 
For rectangular Lawson symmetric surfaces,
the Riemann surface structure is determined up to one real
parameter
(see also the corresponding discussion in \cite{He2}).
In fact,
after a M\"obius transformation,
the antiholomorphic symmetry of the $z$-plane
corresponding to the reflection of $\bbS^3$ can
be taken to be $z\mapsto\bar z$
(so $z_1 = \bar z_0$),
and the branch points can be normalized to
satisfy $|z_0| = |z_1| = 1$.

\subsection{The DPW approach}
Dorfmeister, Pedit and Wu~\cite{DPW} developed a loop group
factorization approach to CMC surfaces in $\bbS^3$ based on
$\lambda$-dependent holomorphic $\mathfrak{sl}(2,\C)$-valued
connection $1$-forms $\eta$, so called {\em DPW potentials}.  This
theory can be seen as a special instance of
theorem~\ref{general_reconstruction}.

In order to construct compact CMC surfaces of higher genus
(numerically) we are going to use a global version of the DPW approach
as put forward in ~\cite{He1,He2}. Because the monodromy
representation of the associated family of flat connections of a
compact CMC surface of genus $g\geq1$ varies non-trivially 
in $\lambda$,
the connection $1$-forms must have poles: on a compact
Riemann surface $M$ the only unitarizable holomorphic connection on
the trivial holomorphic bundle $\underline\C^2\to M$ is the trivial
connection $d$. Thus we need to use meromorphic DPW potentials $\eta$
in order to find families of flat connections
$\hat\nabla^\lambda=d+\eta(\lambda)$ satisfying the conditions of
theorem~\ref{general_reconstruction}.

For an arbitrary closed Riemann surface of genus $g \geq 2$
the general form of a DPW potential for CMC surfaces with
the given Riemann surface structure is unclear.  In the case of
the Lawson surface of genus $2$, the existence and precise form of a
DPW potential up to two unknown functions in $\lambda$ was determined
in~\cite{He2}.  In the slightly more general situation of Lawson
symmetric CMC surfaces of genus $2$ one can prove by the same methods
as in~\cite{He2} that a DPW potential is given by the pullback via $\pi$  of
\begin{equation}
\label{DPW_potential}
\eta = \eta_{A,B} =
\dvector{-\frac{2}{3}\frac{z(2z^2-z_0^2-z_1^2)}{(z^2-z_0^2)(z^2-z_1^2)}
+\frac{A}{z}  
& \lambda^{-1}-\frac{(A+\frac{2}{3})(A-\frac{1}{3})}{B}z^2 \\
 \frac{B}{(z^2-z_0^2)(z^2-z_1^2)}
 -\frac{\lambda A(A+1) z_0^2 z_1^2}{z^2(z^2-z_0^2)(z^2-z_1^2)}
 &\frac{2}{3}\frac{z(2z^2-z_0^2-z_1^2)}{(z^2-z_0^2)(z^2-z_1^2)}
-\frac{A}{z}}dz.
\end{equation}
Here, $A,B$ are $\lambda$-dependent holomorphic functions on a
neighborhood of $\lambda=0$, the Riemann surface $M$ is determined by
\eqref{rss} and $\pi\colon M\to \CP^1$ is the 3-fold covering
branched over $\pm z_0,\pm z_1$. All poles of $d+\pi^*\eta$ are apparent on
$M$, i.e., the local monodromy around every pole is trivial. On the
quotient $M/\Z_3=\CP^1$ the poles at $z=0$ and $z=\infty$ are still
apparent whereas the conjugacy classes of the monodromies around the
branch images $\pm z_0$ and $\pm z_1$ are given by a third root of
unity.

\subsection{The extrinsic closing condition}
The functions $A$ and $B$ in \eqref{DPW_potential} need to be chosen
in such a way that the intrinsic and extrinsic closing conditions
 are satisfied for $d+\eta_{A(\lambda),B(\lambda)}$.  As was
proven in~\cite{He2}, there do not exist finite values for $A$, $B$
and $\lambda$ such that the holonomy of $d+\eta_{A,B}$ is
trivial. Nevertheless, there exist values for $A$ and $B$ such that
the monodromy is upper triangular. This observation yields generalized
extrinsic closing conditions:
 
\begin{Pro}
\label{sym_point_condition}
Consider a Lawson symmetric Riemann surface given by $z_0,z_1\in\C$ and
$\lambda_1,\lambda_2\in \bbS^1\subset\C$. Let $A,B\colon U\to\C$ be
holomorphic functions on an open set $U$ containing the closed unit
disc.  If the connections $d+\eta_{A(\lambda),B(\lambda)}$ have
unitarizable monodromy representation for all $\lambda\in
\bbS^1\setminus\{\lambda_1,\lambda_2\}$ and if the functions $A$ and $B$
are related at the Sym points $\lambda_1$ and $\lambda_2$ by
\begin{equation}
\label{sym_point_condition_1}
B(\lambda_k)=S_k(\lambda_k)\, \, \text{ and }\, \, B'(\lambda_k)=S_k'(\lambda_k)
\end{equation}
where $S_k(\lambda)=z_k^2 \lambda R(\lambda)$ with
$R(\lambda)=A(\lambda)(A(\lambda)-\frac{1}{3})$, then there exists a
closed Lawson symmetric CMC surface in $\bbS^3$ such that the flat
connections $\nabla^\lambda$ of its associated family are gauge
equivalent to $d+\eta_{A(\lambda),B(\lambda)}$ for generic $\lambda\in
U\setminus\{0,\lambda_1,\lambda_2\}$.
\end{Pro}

\begin{proof}
Using the methods of the proof of theorem 4.5 in~\cite{He2} one can
verify that a suitable gauge transformation of the flat connection
$d+\xi_{A(\lambda_k),B(\lambda_k)}$ is a Fuchsian system of upper
triangular form. Moreover, its lower left entry vanishes up to order
$2$ at $\lambda=\lambda_k$. Such a family of Fuchsian systems can be
gauged by a $\lambda$-dependent gauge which is singular at $\lambda_k$
into a family of Fuchsian systems which is diagonal at $\lambda_k$.
The pullback of this diagonal Fuchsian system to $M$ has trivial
monodromy.  Hence, by applying a suitable gauge transformation which
is singular at the Sym points $\lambda=\lambda_1,\lambda_2$ we
obtain from $d+\xi_{A(\lambda),B(\lambda)}$ a family of flat
$\SL(2,\C)$ connections satisfying the conditions of
theorem~\ref{general_reconstruction}.
\end{proof}

\subsection{The intrinsic closing condition}
In contrast to the extrinsic closing condition
the intrinsic closing condition cannot be treated algebraically--- they require a numerical search.
We briefly describe the necessary theory.
 
Since the DPW potential
$d+\eta_{A,B}$ has an apparent singularity at $z=0$
we apply a
lower triangular gauge as in the proof of theorem 4.5 in~\cite{He2} and
obtain a meromorphic potential $d+\tilde\eta_{A,B}$ which is smooth at
$z=0$. Moreover, it satisfies
\[\varphi_2^*(d+\tilde\eta_{A,B})=J(d+\tilde\eta_{A,B})J^{-1}\]
where $\varphi_2$ is given by $z\mapsto-z$ and $J=\text{diag}(i,-i)$ is the
diagonal matrix with entries $i,-i$. This implies that at $z=0$ the
monodromy matrices $M_1$, $M_2$, $M_3$ and $M_4$ around the poles
$z_0$, $z_1$ $-z_0$ and $-z_1$ with respect to the standard basis of
$\C^2$ satisfy
\begin{equation}\label{J_sym}
M_3=JM_1J^{-1}\, \text{ and } \,M_4=JM_2J^{-1}.\end{equation}
All these matrices are in $\SL(2,\C)$ and of trace $-1$ as the
eigenvalues of the residues of the connections $d+\tilde\eta_{A,B}$
are $\pm\frac{1}{3}$. We denote the half-traces of the products by
\[t_{ij} = \tfrac{1}{2}\Tr(M_i M_j).\]
The following proposition gives a
characterization of unitarizable representations for this case:

\begin{Pro}
\label{four_punctured_monodromy_test}
Consider four matrices $M_k\in\SL(2,\C)$ with trace $\text{tr} (M_k)=c\in(-2,2)$ satisfying \eqref{J_sym}
such that they have no
common eigenline.
Then they are simultaneously unitarizable if and
only if $t_{kl}\in[-1,1]$ for all $k,l\in\{1,..,4\}$.
\end{Pro}

\begin{Rem}
We are seeking $\lambda$ families of flat connections which are
unitarizable along the unit circle.  In view of
Remark~\ref{generalized_reconstruction} one can work with the
generalized condition that $t_{kl}(\lambda)\in[-1,1]$ for all
$\lambda\in \bbS^1$ and for all $k,l\in\{1,..,4\}$ if for generic
$\lambda\in \bbS^1$ we have $t_{kl}(\lambda)\in(-1,1)$ for all $k\neq l$.
Then the four matrices are unitarizable by a
diagonal $\lambda$-dependent matrix $D(\lambda)$. In general $D$ has
(finitely many) singularities along the unit circle. By
proposition~\ref{sym_point_condition} and its proof this occurs at
the Sym points.
\end{Rem}

\section{Numerical construction of higher genus CMC surfaces}
\label{sec:experiment}

\subsection{Computing surfaces via DPW}
The surfaces were computed with XLab, a computer framework for surface
theory, experimentation and visualization written in C++ by the second
author.  XLab implements the DPW construction~\cite{DPW} of CMC
surfaces in $\mathbb{S}^3$ from a potential in the following steps.
The holomorphic frame is computed as the numerical solution to an ordinary
differential equation.  Loops appearing in the DPW construction are
infinite dimensional; for computation, they are represented finitely
as Laurent polynomials about $\lambda=0$ by truncation.
Next, the unitary frame
is computed from the holomorphic frame via loop group Iwasawa
factorization.  This calculation applies linear methods to matrices of
coefficients of the Laurent polynomials representing the holomorphic
frame~\cite{KMcIS}.  Next, the CMC immersion is computed by
evaluating the unitary frame at the Sym points.  Finally, the surface
itself is built up by applying its symmetry group to one fundamental
piece computed by the DPW construction.

All Lawson symmetric CMC surfaces were constructed using the following DPW
potential, which is a generalization of the potential~\eqref{DPW_potential}
to arbitrary genus $g$:
\begin{equation}
\label{DPW_g_potential}
\eta = \eta_{A,B} = \dvector{-\frac{g}{g+1}
\frac{z(2z^2-z_0^2-z_1^2)}{(z^2-z_0^2)(z^2-z_1^2)} + \frac{A}{z}  
& \lambda^{-1}
-\frac{\bigl(A+\frac{2}{g+1}\bigr)\bigl(A+\frac{1-g}{1+g}\bigr)}{B}z^2 \\
 \frac{B}{(z^2-z_0^2)(z^2-z_1^2)}
 - \frac{\lambda A(A+1) z_0^2 z_1^2}{z^2(z^2-z_0^2)(z^2-z_1^2)}
 &\frac{g}{g+1}\frac{z(2z^2-z_0^2-z_1^2)}{(z^2-z_0^2)(z^2-z_1^2)}
-\frac{A}{z}}dz
.
 \end{equation}

The Riemann surface structure of genus $g$ is given by
$y^{g+1}=\frac{z^2-z_0^2}{z^2-z_1^2}$, while $A,\,B$ are the accessory
parameters.  The extrinsic closing condition at the Sym points is
similar to that of the case $g=2$ described in proposition~\ref{sym_point_condition},
with the function $R$ replaced by
$\tilde R=A(A+\frac{1-g}{1+g})$.  
%
For rectangular Lawson symmetric CMC surfaces, the reflection symmetry
is imposed by stipulating
$\overline{\eta(\lambda,\,z)} = \eta(\overline{\lambda},\,\overline{z})$,
or equivalently that $A$ and $B$ are real for real $\lambda$,
and the Sym points satisfy $\lambda_2=\overline{\lambda_1}$.

\subsection{Computing the accessory parameters}

The computationally intensive step in the numerical construction of
Lawson symmetric CMC surfaces is the search for the accessory
parameters $A$ and $B$ in the DPW potential which satisfy the
intrinsic and extrinsic closing conditions, thereby making
unitarizable monodromy and trivial monodromy at the Sym points,
thereby closing the surface periods.  The infinite dimensional space
of accessory parameters is approximated by power series truncation.
The accessory parameters are computed by optimization (minimization)
algorithms.  The function to be minimized (objective function) is
constructed to measure how far a set of generators of the monodromy
group of the DPW potential is from being simultaneously
unitarizable. This measure is computed as the average over a set of
equally spaced sample points on the unit circle in the spectral plane.
To speed up these lengthy calculations, the objective function is
computed in parallel over the sample points simultaneously on a
multicore computer.  Once the accessory parameters in the DPW
potential are found, the diagonal unitarizer, computed as
in~\cite{KS13}, is used as the initial value for the holomorphic frame
so that the monodromy is unitary.

The monodromies were computed with an adaptive Runge-Kutta-Fehlberg 4(5) ODE solver with
absolute and relative error tolerances set to $10^{-12}$,
and the minimization was performed with the gradient-free
optimization algorithms PRAXIS, BOBYQA and Nelder-Mead,
chosen heuristically.

To compute the holomorphic accessory parameters $A$ and $B$,
expand at $\lambda=0$
\[
A=\sum_{k=0}^\infty a_k\lambda^k
\quad\text{and}\quad
B=\sum_{k=0}^\infty a_k\lambda^k
\]
and for a finite $N$,
approximate by truncation
\[
A^N=\sum_{k=0}^N a_k\lambda^k
\quad\text{and}\quad
B^N=\sum_{k=0}^N a_k\lambda^k.
\]

Based on proposition~\ref{four_punctured_monodromy_test},
we define the \emph{objective function} $\mathcal{F}$
as a non-negative function measuring how close the monodromy
is from being unitarizable
and hence how close the surface is to being closed,
with $\mathcal{F}=0$ representing exact unitarizability.
More precisely, $\mathcal{F}$ is the root mean square of the imaginary parts of the monodromy
half-traces, penalized when their real parts fail to be in $[-1,\,1]$.
To define $\mathcal{F}$, consider
$\mathcal{F}_1\colon \mathbb{S}^1\times \C^N\times \C^N\to \R$ measuring
unitarizablity at one point $\mu\in\bbS^1$
\[
\mathcal{F}_1(\mu, a_0,..,a_N,b_0,..,b_N)=
\tfrac{1}{4}\left(\chi(t_{12}) + \chi(t_{13}) + \chi(t_{23}) + \chi(t_{24})\right),
\]
where $t_{ij}=\frac{1}{2}\Tr(M_i M_j)$
($i,\,j\in\{1,\dots,4\}$)
for the monodromy matrices $M_i$
of the connection
$\nabla=d+\eta_{A(\mu),B(\mu)}$ on the four-punctured sphere
$\CP^1\setminus\{\pm z_0,\pm z_1\}$ and
$\chi\colon\C\to \R$ given by
\[
\chi(t) = {(\Im t)}^2 + {(\max(0,\,|\Re t|-1))}^2
\]
is the squared distance from $t$ to the line segment
$[-1,\,1]$ in the complex plane.
Then
$\mathcal{F}\colon\C^N\times \C^N\to \R$ is defined by
\[\mathcal{F}(a_0,\dots,a_N,b_0,\dots,b_N)=
\sqrt{\textstyle
\frac{1}{K}\sum_{k=1}^K \mathcal F_1(\mu_k,a_0,\dots,a_N,\,b_0,\dots,b_N)}
\]
where $\mathcal{F}_1$ is evaluated
at a finite number of equally spaced sample points
$\mu_1,\dots,\mu_K\in\mathbb{S}^1$.
In the case of rectangular Lawson symmetric surfaces,
by the antiholmorphic symmetry it suffices to take 
$\mu_1,\dots,\mu_K$ along the upper half of $\bbS^1$.

The extrinsic closing conditions~\eqref{sym_point_condition_1} in
proposition~\ref{sym_point_condition} for Sym points
$\lambda_1,\lambda_2\in \mathbb{S}^1$, is enforced in the numerical
search by replacing $B$ by $C$ via
\begin{equation}
\label{B}
B=f R+h C,
\end{equation}
where $R$ is defined in proposition~\ref{sym_point_condition},
$f$ is the unique polynomial of degree $\leq3$ satisfying
\[
f(\lambda_1)=z_0^2\lambda_1,\quad
f(\lambda_2)=z_1^2\lambda_2,\quad
f'(\lambda_1)=z_0^2,\quad
f'(\lambda_2)=z_1^2
\]
and
\[
h(\lambda)=(\lambda-\lambda_1^2)(\lambda-\lambda_2^2).
\]
The coefficients of $A$ and $C$ truncated to 
\[
C^N=\sum_{k=0}^{N}c_k\lambda^k
\]
are the free variables in the numerical search.

\subsection{The Lawson surfaces $\xi_{g,1}$}
\label{Lawson2}

An initial test of our experimental setup was to compute the Lawson
surfaces $\xi_{g,1}$ for $g=1,\dots,6$
(figure~\ref{fig:lawson}).  The case genus $g=1$ recovers the
Clifford torus.  For genus $g=2$, the numerically computed surface
has the symmetries $\varphi_2,$
$\varphi_3$ and $\tau$. Moreover, it has an additional space
orientation reversing and an anti-holomorphic symmetry. The first
symmetry is induced by the property that $A$ and $B$ are even
functions of $\lambda$ while the later symmetry follows from the
reality of $A$ and $B.$ It then can be deduced that the minimal
surface in $\bbS^3$ must be the Lawson surface $\xi_{2,1}$.  Note that an
energy formula (analogous to the energy formula in~\cite{He3}) applied
to our numerical DPW potential yields an area of $21.91,$ a value
which only slightly differs from what has been numerically computed
in~\cite{HKS} using the Willmore flow.

\subsection{Family I: CMC deformations of the Lawson surfaces $\xi_{g,1}$}
\label{sec:family1-experiment}

\begin{figure}
\centering
\includegraphics[width=0.45\textwidth]{lawson2.pdf}
\includegraphics[width=0.45\textwidth]{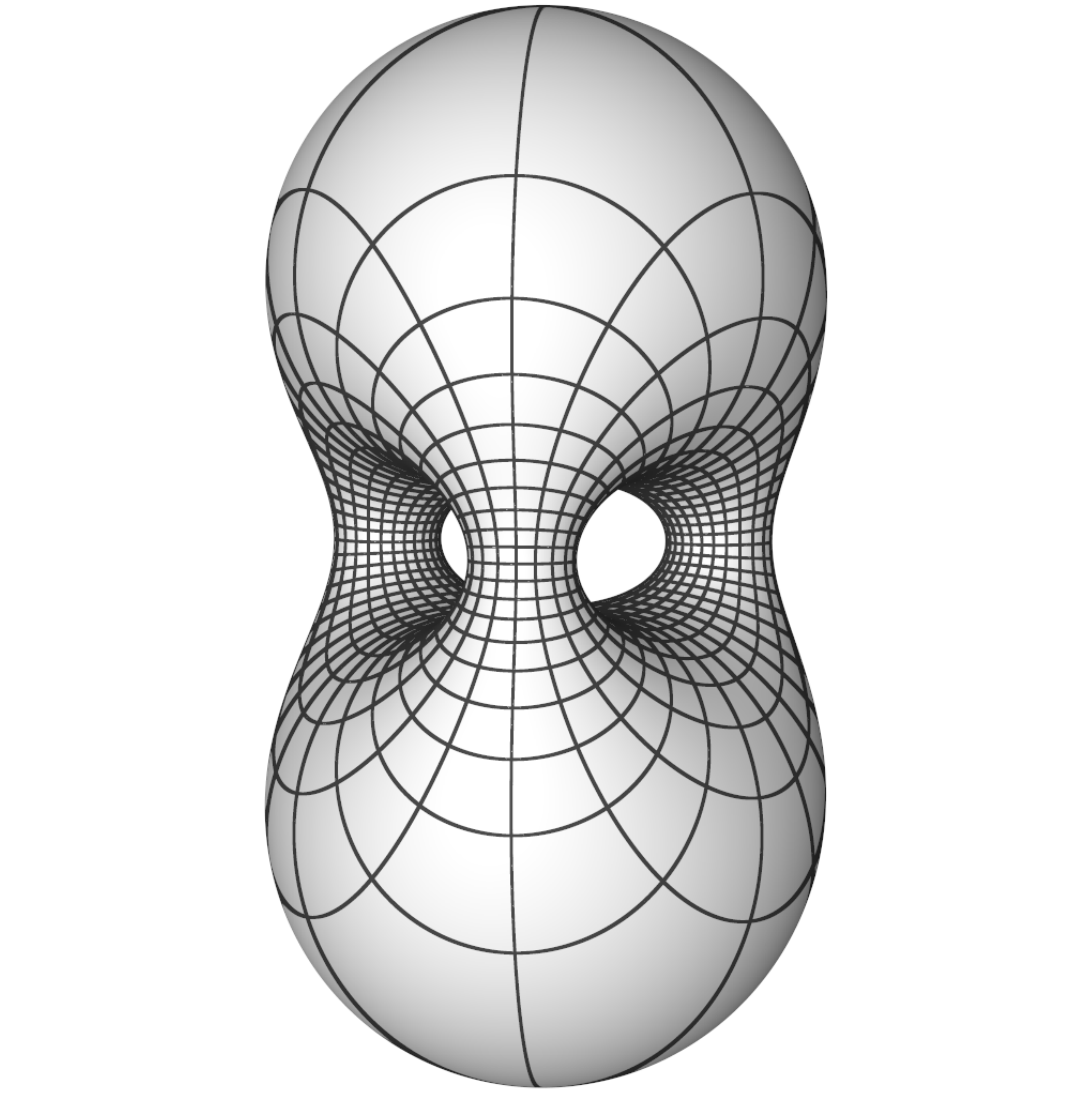}
\includegraphics[width=0.45\textwidth]{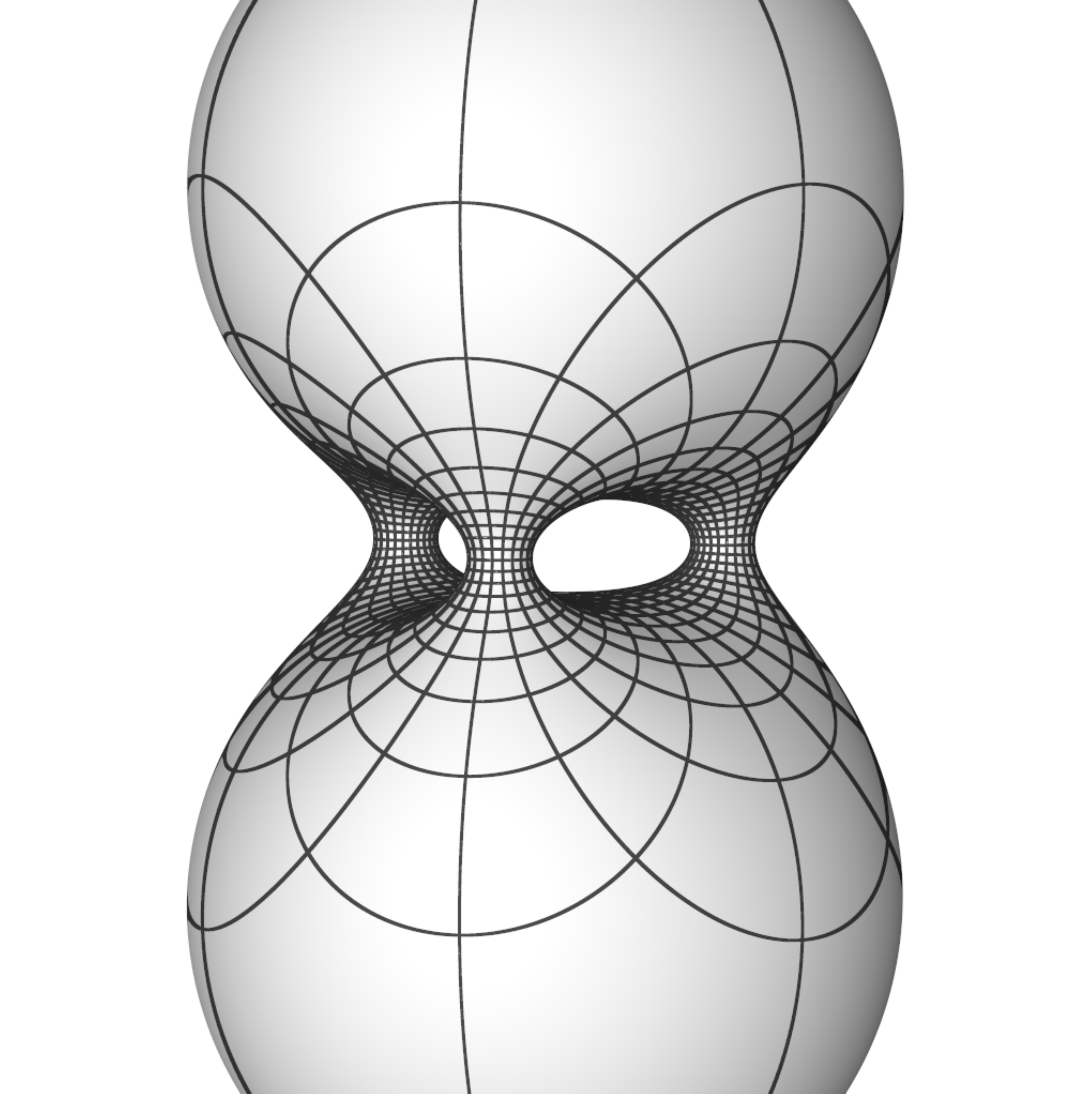}
\includegraphics[width=0.45\textwidth]{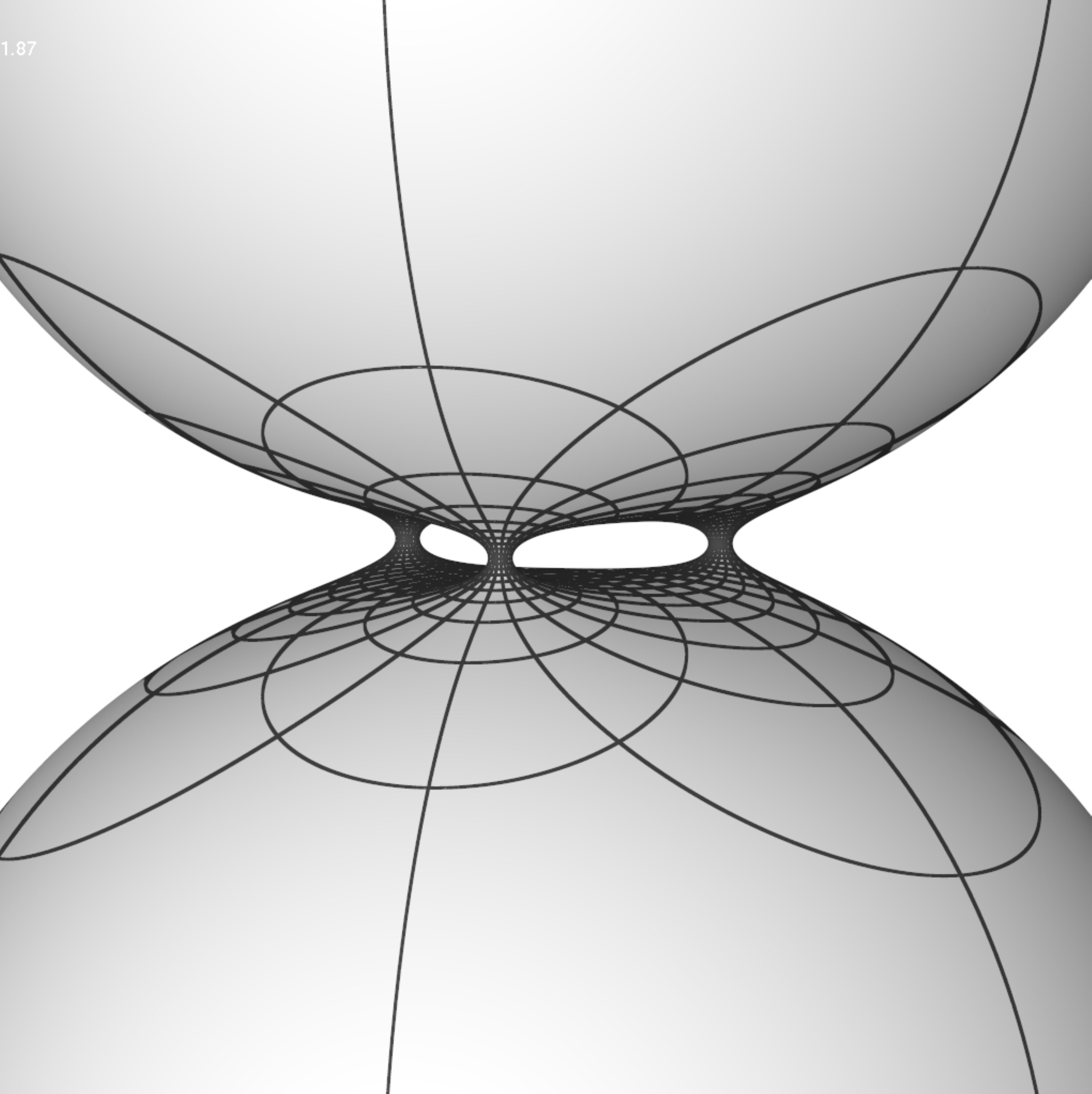}

\caption{
\footnotesize
Family I of CMC surfaces of genus 2 in $\bbS^3$, starting at the Lawson surface
(upper left). As the necks shrink as the lobes expand,
the surfaces converge to a double covered minimal sphere as singular limit.
}
\label{fig:lawsoncmc1}
\end{figure}

In designing deformations from the Lawson surfaces, physical intuition
suggests that the pressure differential inside and outside the CMC
surface breaks the ambient space orientation-reversing symmetries.  In
terms of the DPW potential, the accessory parameters $A(\lambda)$ and
$B(\lambda)$ are no longer even functions of $\lambda$, as can be
deduced from the Sym point condition \eqref{sym_point_condition_1} in
proposition~\ref{sym_point_condition}.

In the case of genus $g=1$ we reconstructed via the DPW
potential~\eqref{DPW_g_potential} the $1$-parameter family of
homogeneous tori of spectral genus $0$, starting at the Clifford
torus, and its bifurcations into CMC tori of spectral genus
$1$~\cite{KSS}.  Along this family are bifurcations to the Delaunay
tori of spectral genus $1$.  The bifurcation appears in our setup as
follows.  The zero $\lambda_0>1$ of $B$ which is closest to the origin
is of order $2$ for the Clifford torus. As it crosses the unit circle
it can continue either as a double zero to the inside or bifurcate to
two simple zeros reflected across the unit circle. When it continues
as a double zero the CMC tori remain homogeneous whereas in the second
case one obtains unduloidal rotational Delaunay tori of spectral genus
$1$.

In the case of genus $g=2$, we numerically searched for and found
minimizers of $\mathcal F$ starting with the initial data of the
Lawson surface on a slightly deformed rectangular Lawson symmetric
Riemann surface.
Repeating, we obtained a family of Lawson symmetric CMC surfaces
through the Lawson surface $\xi_{2,1}$.  This family of Lawson symmetric CMC
surfaces is shown in figure~\ref{fig:lawsoncmc1}.

We further computed the deformation family from $\xi_{3,1}$ and expect that such a 1-parameter deformation of Lawson
symmetric CMC surfaces exists for each genus realizing the expected 1-dimensional CMC deformation family 
of $\xi_{g,1}$~\cite{KMP}.

\subsection{Family II: a flow through CMC surfaces of genus $2$}
\label{sec:family2-experiment}

\begin{figure}
\centering
\includegraphics[width=0.475\textwidth]{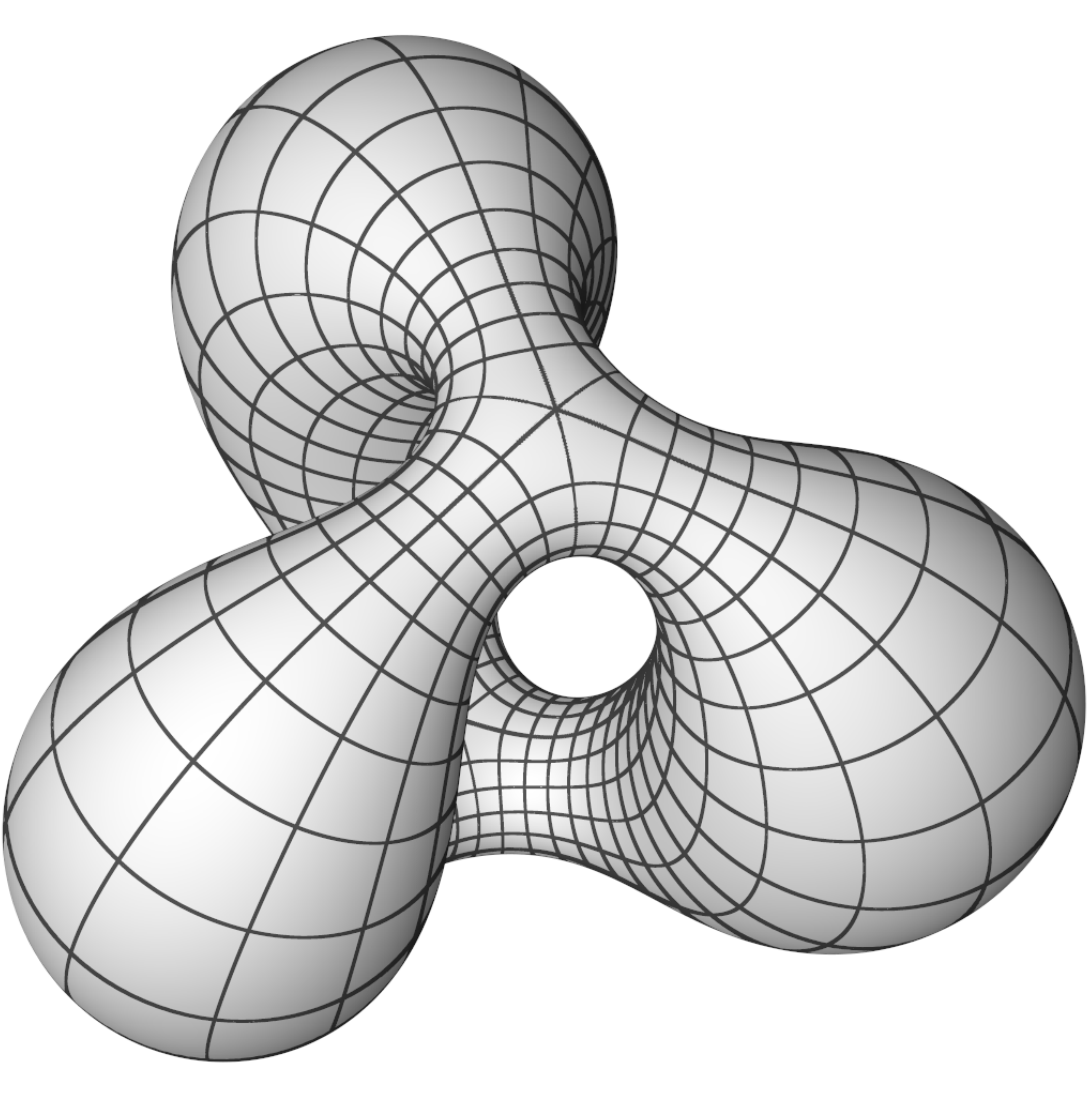}
\includegraphics[width=0.475\textwidth]{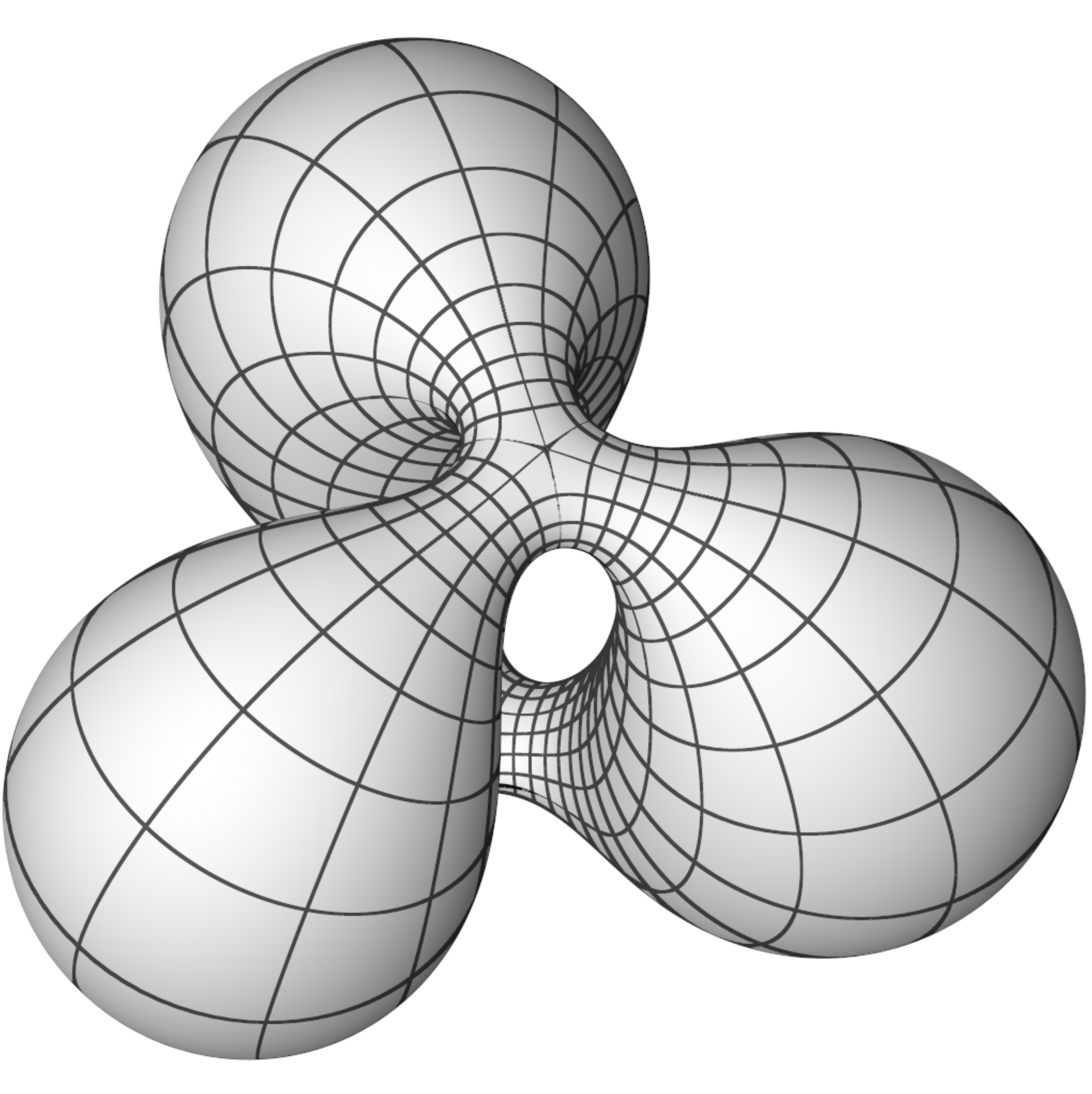}
\includegraphics[width=0.475\textwidth]{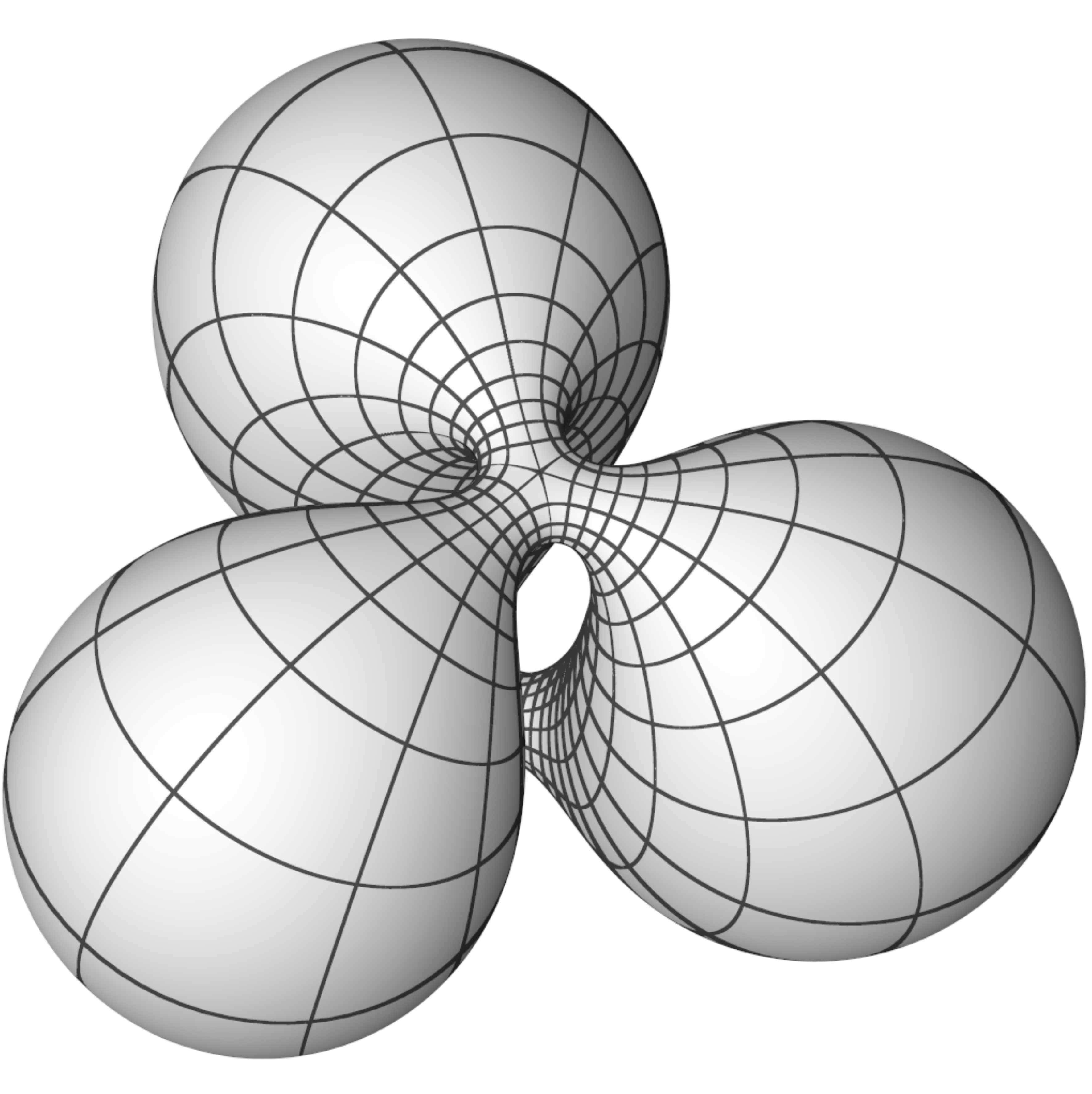}
\includegraphics[width=0.475\textwidth]{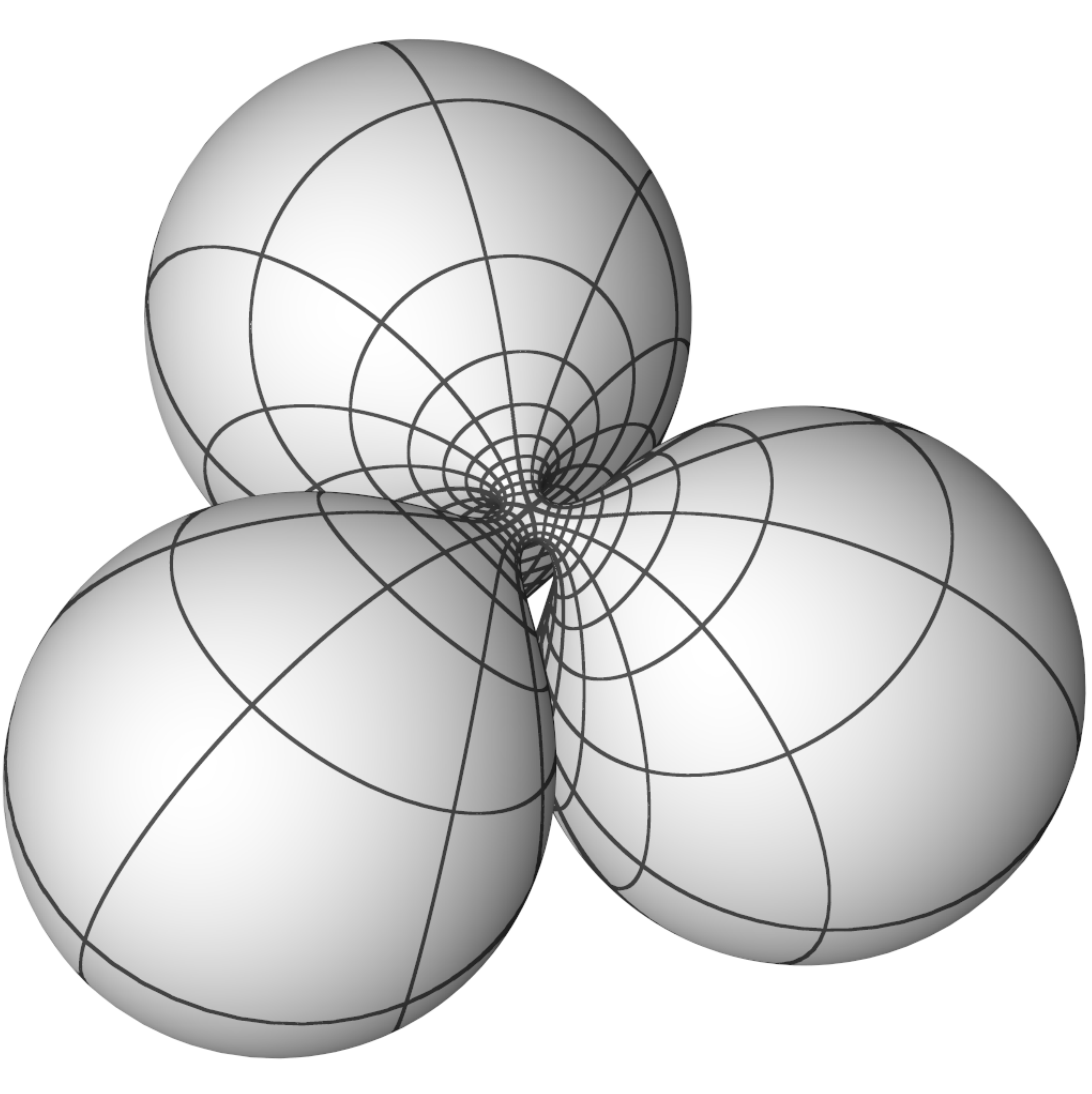}
\caption{
\footnotesize
Family II of CMC surfaces in $\bbS^3$ of genus $2$.
This stereographic projection exposes their $\Z_3$ symmetry.
Unlike family I, this family is not connected to the Lawson minimal
surface $\xi_{2,1}$, 
The flow is conjectured to limit to a necklace of three CMC spheres
(lower right) as the conformal type degenerates.
}
\label{fig:lawsoncmc2a}
\end{figure}

We found an additional $1$-parameter family of
rectangular Lawson symmetric CMC surfaces by
modifying the objective function $\mathcal{F}$ to enforce the
condition that $B$ and $A+1$ have a common zero $\lambda_0$ inside the
unit disk. A detailed discussion of this condition can be found in
section~\ref{prop:unstable}.

Figure~\ref{fig:lawsoncmc2a} shows a sequence in this family which
converges to a singular limit which is a chain of three round CMC
spheres in $\bbS^3$.  Figure~\ref{fig:lawsoncmc2b} depicts a
different stereographic projection of same family, in which the
symmetries $\varphi_2$ and $\tau$ appear as symmetries of Euclidean
$3$-space.  In this view the surfaces appear as 2-lobed Delaunay tori
into which a piece of a Delaunay cylinder is glued.

Figure~\ref{fig:graph} plots the
the Riemann surfaces structures and the mean curvature of the numerically
computed families I and II of Lawson symmetric CMC surfaces of genus $2$.

\begin{figure}
\centering
\includegraphics[width=0.475\textwidth]{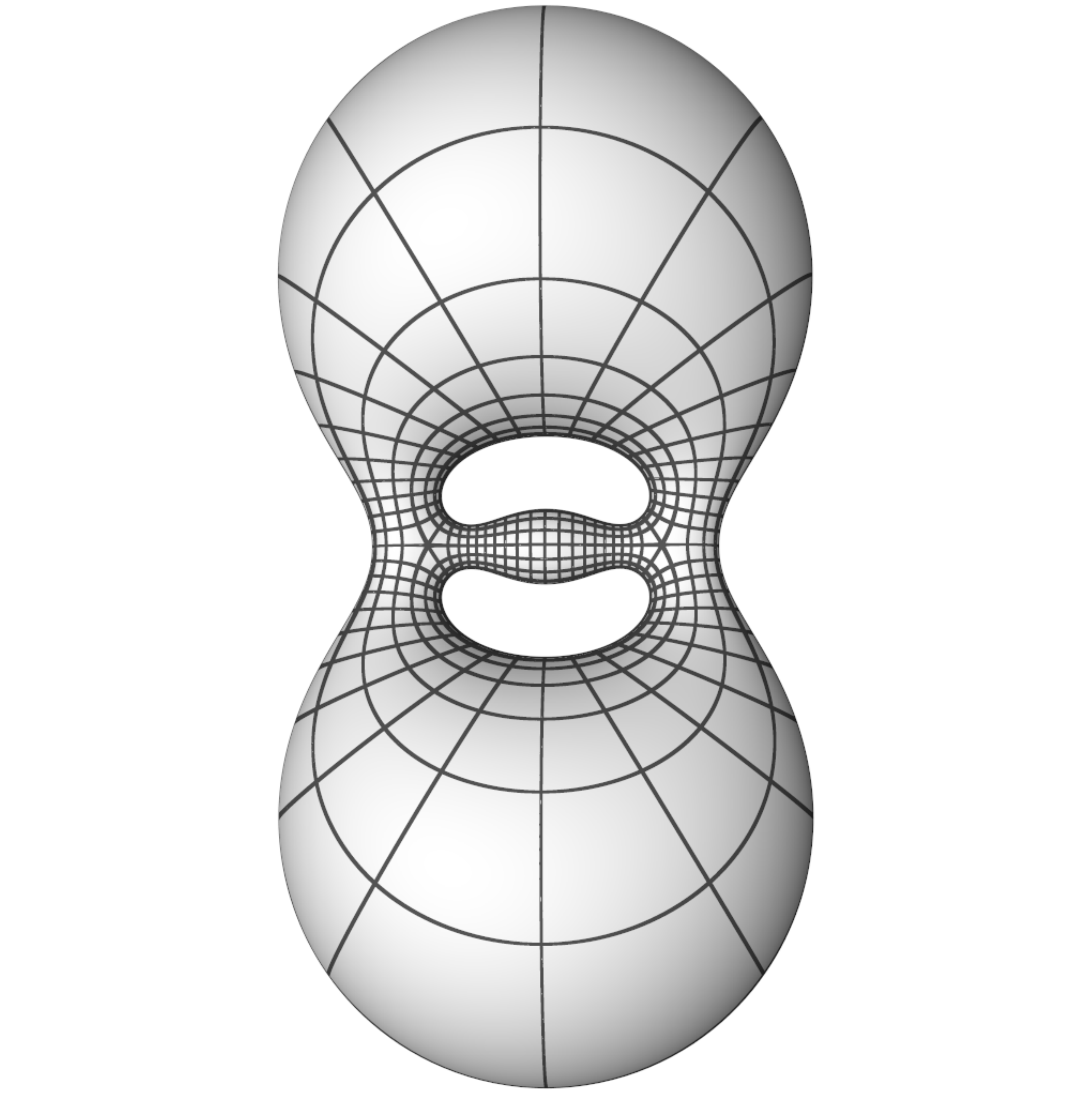}
\includegraphics[width=0.475\textwidth]{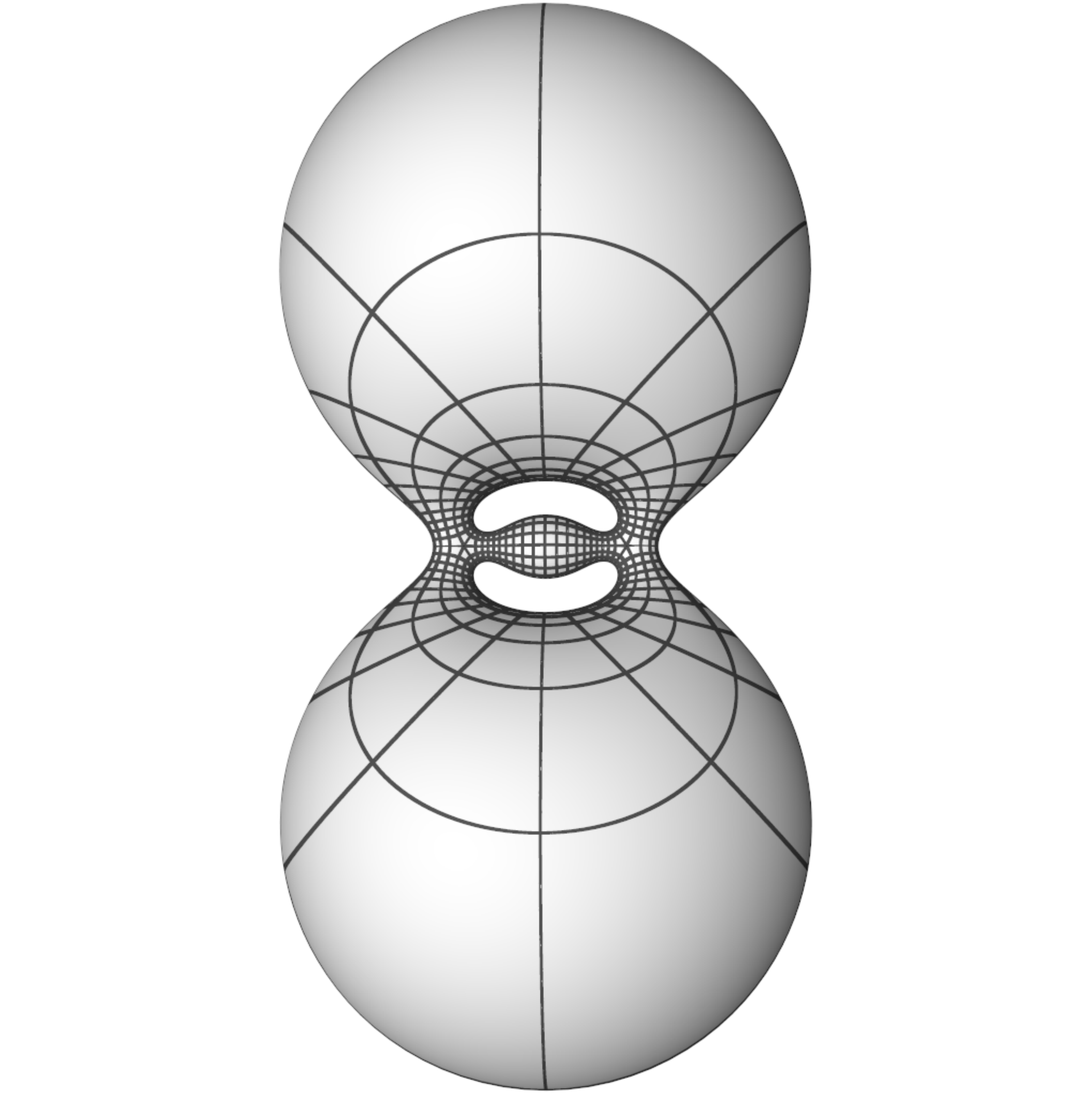}
\includegraphics[width=0.475\textwidth]{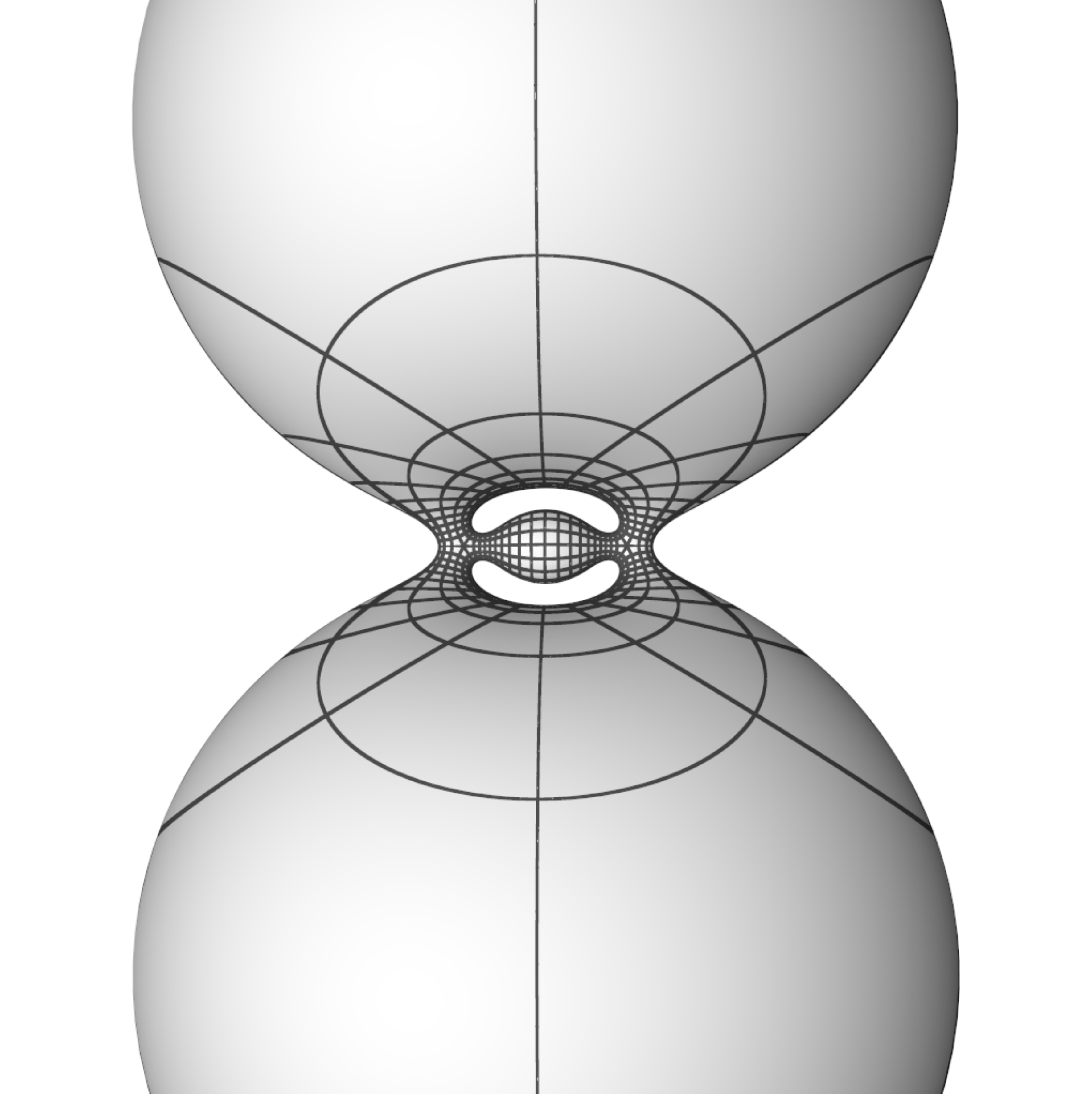}
\includegraphics[width=0.475\textwidth]{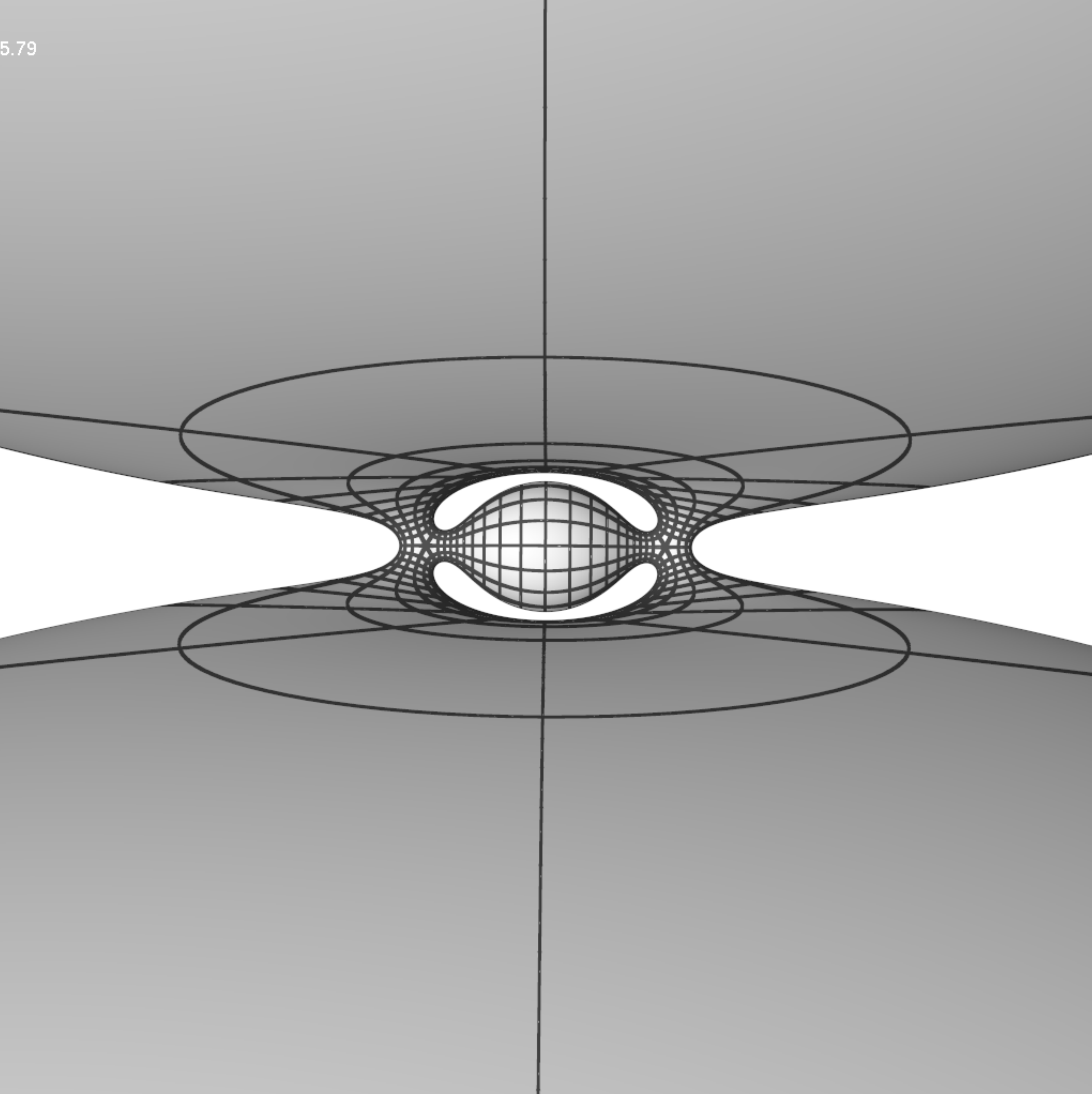}
\caption{
\footnotesize
A different stereographic projection of genus $2$ CMC surfaces of family II.
In this view, the flow can be seen to mimic that of the $2$-lobe
Delaunay tori, but with a piece of a Delaunay cylinder glued in.
}
\label{fig:lawsoncmc2b}
\end{figure}
\begin{figure}
\centering
\includegraphics[width=0.495\textwidth]{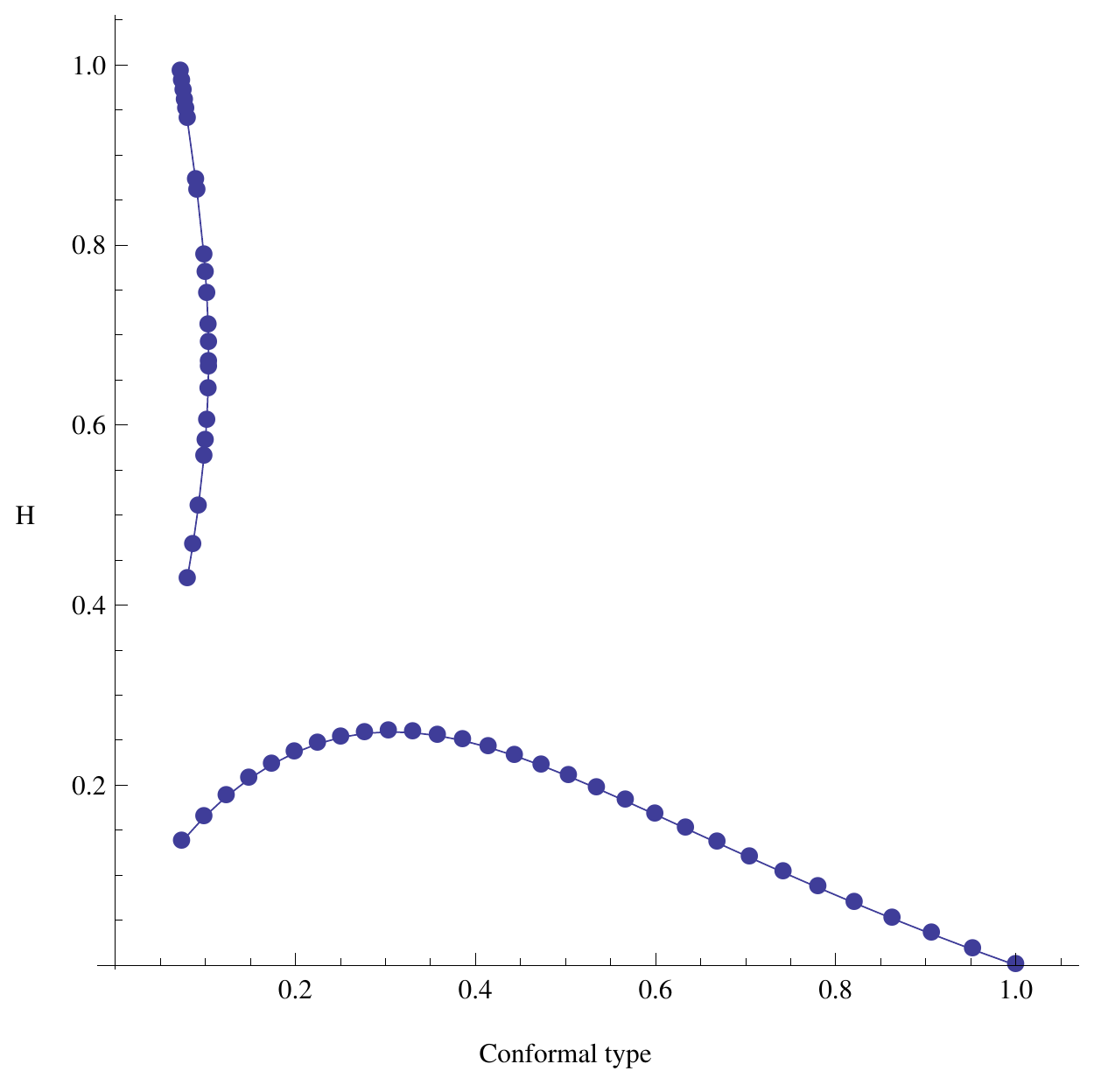}
\caption{
The two families I and II of numerically computed CMC surfaces in $\bbS^3$ of genus $2$,
plotting conformal type against mean curvature.
The horizontal curve represents Family I, starting at the Lawson surface (lower right)
and limiting to a doubly covered minimal $2$-sphere at the origin (lower left).
The vertical curve is Family II, conjectured to limit to a necklace of three spheres
as the mean curvature decreases.
}
\label{fig:graph}
\end{figure}

\section{Outlook}\label{outlook}
In our experiments we have constructed two families of compact
genus $2$ CMC surfaces in the $3$-sphere.
While the experiments not only give evidence for surfaces families,
but hint toward a notion of a generalized spectral curve for compact
CMC surfaces.
We propose a notion of spectral genus in terms
of the parabolic structures.

\subsection{Parabolic structures}
A parabolic structure, which prescribes boundary data to a flat connection
on a punctured Riemann surface, is an analog of
the induced holomorphic structure of a flat connection
on a compact surface;
in fact, there is a strong relationship between the two notions~\cite{Bis4}.
For our purposes it is enough to consider
parabolic structures for Fuchsian systems, that is,
meromorphic connections on a punctured $\mathbb{P}^1$ with first
order poles.  For details see~\cite{MS, Bis4}.

The general form of a $\mathfrak{sl(2,\C)}$ Fuchsian system is
\begin{equation}
\nabla=d+\sum_{i=1}^n A_n \frac{dz}{z-z_n}
\end{equation}
with residues $A_1,\dots,A_n\in\mathfrak{sl(2,\C)}$, and pairwise distinct
singularities $z_1,\dots,z_n\in\mathbb{C}$
of $\nabla$.
We assume $z=\infty$ is not a singularity, so $\sum A_i=0$,
and specialize to the case that each $A_i$ has an eigenvalue
$\rho_i\in]0,\frac{1}{2}[$, with corresponding eigenline $E_i$.
The parabolic structure of $\nabla$ is then given by the
filtrations
\[
0\subset E_i\subset \C^2
\]
over the singularities together with the corresponding weight
filtrations $(\rho_i,-\rho_i)$: the line $E_i$ is equipped with
the weight $\rho_i$, while $\C^2/ E_i$
is equipped with the
weight $-\rho_i$.
The parabolic degree of a holomorphic line subbundle $L\subset \underline \C^2$
is
\[\pdeg L=\deg L+\sum_i\gamma_i,\]
where $\gamma_i$ is defined to be $\rho_i$ if
$L_{p_i}=E_i$ and $-\rho_i$ otherwise ($i=1,\dots,n$).

If $\nabla$ has a unitarizable
monodromy representation, then the parabolic degree for every
holomorphic line subbundle $L\subset \underline \C^2$ is non-positive,
and is negative if $\nabla$ is irreducible. Parabolic structures
satisfying this property are called semi-stable respectively
stable. 
Conversely, by Mehta and Seshadri~\cite{MS} every stable parabolic structure 
is determined by a unitary representation of the fundamental group of the punctured Riemann surface.
Hence, every stable parabolic structure as defined above admits a
Fuchsian system with unitarizable monodromy representation, see also \cite{Biq} for a more differential-geometric approach.

\subsection{Unstable parabolic structures of Lawson symmetric CMC surfaces}
Consider the DPW potential $\eta$~\eqref{DPW_potential} for
a genus $2$ Lawson symmetric CMC surface.
As in~\cite{He2}, $\eta$ can be gauged to a $\lambda$-family of Fuchsian systems.
The induced parabolic structures at $\lambda\in\Cstar$ extend holomorphically
to $\lambda=0$.
The parabolic structure of a Lawson symmetric CMC surface is stable for generic $\lambda$,
since the parabolic structures
for $\lambda$ on the unit circle are generically stable,
and stability is an open condition.
The following proposition characterizes the unstable parabolic structures:
\begin{Pro}
\label{prop:unstable}
The parabolic structure induced by $d+\eta_{A(\lambda_0),B(\lambda_0)}$
for $\lambda_0\in\Cstar$
is unstable if and only if
$B$ and $A+1$ vanish at $\lambda_0$, with orders satisfying
$ord_{\lambda_0} B\leq ord_{\lambda_0}(A+1)$.
\end{Pro}
\begin{proof}
If $B(\lambda_0)=0$ and $A(\lambda_0)=-1$ then $d+\eta$ can be gauged
to the Fuchsian system
\[
d+\dvector{
\frac{2}{3}\frac{z(2 C z_0^2z_1^2+z_0^2+z_1^2)}{(z^2-z_0^2)(z^2-z_1^2)} &
 -\frac{16}{9}\frac{z_0^2z_1^2(z^2 C+1)}{(z^2-z_0^2)(z^2-z_1^2)}
\\  \frac{z^2+C z_0^2 z_1^2}{(z^2-z_0^2)(z^2-z_1^2)}&
 -\frac{2}{3}\frac{z(2 C z_0^2z_1^2+z_0^2+z_1^2)}{(z^2-z_0^2)(z^2-z_1^2)}}dz
\]
where $C=\frac{A(\lambda_0)+1}{B(\lambda_0)}\in\C$ by assumption.  The
eigenlines of the residues of this Fuchsian system at the
singularities $z=\pm z_0,\pm z_1$ with respect to the eigenvalue
$\frac{1}{3}$ are $\C\dvector{\tfrac{4}{3} z\\ 1}$. Hence
they coincide with the fibers of a line subbundle of degree $-1$ at
$z=\pm z_0,\pm z_1$, so the induced parabolic structure is unstable.
The converse, the proof of which would lead too far afield,
can be deduced from results in~\cite{He4}.
\end{proof}

\subsection{Generalized spectral genus for Lawson symmetric CMC surfaces of genus $2$}\label{gen_spec_gen}

Proposition~\ref{prop:unstable} provides a means to distinguish the two
experimental Lawson symmetric CMC families.
For family I, at the Lawson surface $\xi_{2,1}$ and surfaces nearby,
$B$ does not vanish in the unit disk.
As the family progresses from the Lawson surface
to the doubly covered minimal sphere as singular limit,
a first order zero $\lambda_0$ of $B$ crosses the
unit circle into the unit disk and approaches $\lambda=0$ (figure~\ref{fig:lawsoncmc1}).
Since for this family $A(\lambda_0)=-\tfrac{2}{3}\neq-1$,
the corresponding parabolic
structures at $\lambda_0$ are stable by proposition~\ref{prop:unstable}.

In fact it can be shown that for each member of family I,
for spectral parameter $\lambda$ inside the unit disk
the corresponding parabolic structures is stable or strictly semi-stable.
The proof of this relies on the fact that within families of Lawson symmetric flat
connections, unstable parabolic structures cannot disappear,
but a detailed proof of this fact would need a careful treatment of the
moduli spaces of flat connections and parabolic structures~\cite{He3, He4}.
On the other hand, for all members of family II, as noted in section~\ref{sec:experiment},
$B$ and $A+1$ have a single common zero $\lambda_0$ of first order inside the
unit disk, so the corresponding parabolic structure at $\lambda_0$ is unstable
by proposition~\ref{prop:unstable}.
As unstable parabolic
structures do not admit unitarizable flat connections, any $\lambda_0$
at which the corresponding parabolic structure is unstable cannot cross the unit circle, 
where the connections are unitarizable.
Hence the number of points $\lambda_0$ in the unit disk
at which the corresponding parabolic structure is unstable
is an invariant of a family of genus $2$ Lawson symmetric CMC surfaces.

This shows that the two families
can be distinguished
by the number of unstable parabolic structures inside the unit disk,
$0$ for family I and $1$ for family II.
While this number might be the only integer invariant of
(Lawson symmetric) families of CMC surfaces of genus $2$, the
authors expect that in the study of the moduli space of compact CMC surfaces,
the number of strictly semi-stable parabolic structures inside the unit disk
will also play a role. In fact, the
existence of a pair of isomorphic strictly semi-stable parabolic
structures at $\lambda_1$ and $1/{\overline{\lambda_1}}$ reflected across
the unit circle would imply the existence of parallel line subbundles at $\lambda_1$ and $1/{\overline{\lambda_1}},$
yielding the existence of a nontrivial
isospectral deformation of the CMC surface, see~\cite{He3} and references therein. For none of our numerically computed surfaces
does there exist a parallel line subbundle for a spectral parameter $\lambda$ away from the unit circle. Therefore, these 
surfaces cannot have any
isospectral deformations.


\begin{thebibliography}{10}

\bibitem{AL}
B.~Andrews and H.~Li, \emph{Embedded constant mean curvature tori in the
  three-sphere}, arXiv:1204.5007v3.

\bibitem{Biq}
O.~Biquard, \emph{Fibr\'es paraboliques stables et connexions singuli\`eres
              plates}, Bull. Soc. Math. France  \textbf{119} (1991), 231--257.

\bibitem{Bis4}
I.~Biswas, \emph{Parabolic bundles as orbifold bundles}, Duke Math. J.
  \textbf{88} (1997), no.~2, 305--325.

\bibitem{Bo}
A.~I. Bobenko, \emph{Surfaces of constant mean curvature and integrable
  equations}, Uspekhi Mat. Nauk \textbf{46} (1991), no.~4(280), 3--42, 192.

\bibitem{B}
S.~Brendle, \emph{Embedded minimal tori in {$S^3$} and the {L}awson
  conjecture}, Acta Math. \textbf{211} (2013), no.~2, 177--190.

\bibitem{DPW}
J.~Dorfmeister, F.~Pedit, and H.~Wu, \emph{Weierstrass type representation of
  harmonic maps into symmetric spaces}, Comm. Anal. Geom. \textbf{6} (1998),
  no.~4, 633--668.

\bibitem{DW}
J.~Dorfmeister and H.~Wu, \emph{Unitarization of loop group representations of
  fundamental groups}, Nagoya Math. J. \textbf{187} (2007), 1--33.

\bibitem{He4}
L.~Heller and Heller. S., \emph{Abelianization of fuchsian systems on a
  4-punctured sphere and applications}, arXiv:1404.7707v2.

\bibitem{He3}
S.~Heller, \emph{A spectral curve approach to lawson symmetric cmc surfaces of
  genus $2$}, to appear {\em Math. Ann}.

\bibitem{He1}
\bysame, \emph{Higher genus minimal surfaces in {$S^3$} and stable bundles}, J.
  Reine Angew. Math. \textbf{685} (2013), 105--122.

\bibitem{He2}
\bysame, \emph{Lawson's genus two surface and meromorphic connections}, Math.
  Z. \textbf{274} (2013), no.~3-4, 745--760.

\bibitem{Hi2}
N.~J. Hitchin, \emph{Harmonic maps from a {$2$}-torus to the {$3$}-sphere}, J.
  Differential Geom. \textbf{31} (1990), no.~3, 627--710.

\bibitem{HKS}
L.~Hsu, R.~Kusner, and J.~Sullivan, \emph{Minimizing the squared mean curvature
  integral for surfaces in space forms}, Experiment. Math. \textbf{1} (1992),
  no.~3, 191--207.

\bibitem{K1}
N.~Kapouleas, \emph{Compact constant mean curvature surfaces in {E}uclidean
  three-space}, J. Differential Geom. \textbf{33} (1991), no.~3, 683--715.

\bibitem{K2}
\bysame, \emph{Constant mean curvature surfaces constructed by fusing {W}ente
  tori}, Invent. Math. \textbf{119} (1995), no.~3, 443--518.

\bibitem{KPS}
H.~Karcher, U.~Pinkall, and I.~Sterling, \emph{New minimal surfaces in
  {$S^3$}}, J. Differential Geom. \textbf{28} (1988), no.~2, 169--185.

\bibitem{KMcIS}
M.~Kilian, I.~McIntosh, and N.~Schmitt, \emph{New constant mean curvature
  surfaces}, Experiment. Math. \textbf{9} (2000), no.~4, 595--611.

\bibitem{KS}
M.~Kilian and M.~Schmidt, \emph{On the moduli of constant mean curvature
  cylinders of finite type in the 3-sphere}, arXiv:0712.0108v2, 2008.

\bibitem{KSS}
M.~Kilian, M.~U. Schmidt, and N.~Schmitt, \emph{Flows of constant mean
  curvature tori in the 3-sphere: the equivariant case}, arXiv:1011.2875v2,
  2010.

\bibitem{KS13}
M.~Kilian and N.~Schmitt, \emph{Constant mean curvature cylinders with
  irregular ends}, J. Math. Soc. Japan \textbf{65} (2013), no.~3, 775--786.

\bibitem{KMP}
R.~Kusner, R.~Mazzeo, and D.~Pollack, \emph{The moduli space of complete
  embedded constant mean curvature surfaces}, Geom. Funct. Anal. \textbf{6}
  (1996), no.~1, 120--137.

\bibitem{L}
H.~B. Lawson, Jr., \emph{Complete minimal surfaces in {$S^{3}$}}, Ann. of Math.
  (2) \textbf{92} (1970), 335--374.

\bibitem{MS}
V.~B. Mehta and C.~S. Seshadri, \emph{Moduli of vector bundles on curves with
  parabolic structures}, Math. Ann. \textbf{248} (1980), no.~3, 205--239.

\bibitem{PS}
U.~Pinkall and I.~Sterling, \emph{On the classification of constant mean
  curvature tori}, Ann. of Math. (2) \textbf{130} (1989), no.~2, 407--451.

\bibitem{SKKR}
N.~Schmitt, M.~Kilian, S.-P. Kobayashi, and W.~Rossman, \emph{Unitarization of
  monodromy representations and constant mean curvature trinoids in
  3-dimensional space forms}, J. Lond. Math. Soc. (2) \textbf{75} (2007),
  no.~3, 563--581.

\end{thebibliography}
\end{document}